\documentclass[12pt]{amsart}

\usepackage{dsfont}
\usepackage{cancel}
\usepackage{mathtools}
\mathtoolsset{showonlyrefs=true}
\usepackage[hmargin=0.8in,height=8.6in]{geometry}
\usepackage{amssymb,amsthm, times}
\usepackage{delarray,verbatim}
\usepackage{ifpdf}
\ifpdf
\usepackage[pdftex]{graphicx}
 \else
\usepackage[dvips]{graphicx}
 \fi

\usepackage{hyperref}
\hypersetup{colorlinks=true,linkcolor=myblue,citecolor=blue!50,urlcolor=blue}

\usepackage{bm}

\usepackage{ifpdf}
\usepackage{color}
\definecolor{webgreen}{rgb}{0,.5,0}
\definecolor{webbrown}{rgb}{.8,0,0}
\definecolor{emphcolor}{rgb}{0.5,0.95,0.95}

\usepackage{hyperref}
\hypersetup{%
	colorlinks=true,
	linkcolor=webbrown,
	filecolor=webbrown,
	citecolor=webgreen,
	breaklinks=true}
\ifpdf \hypersetup{pdftex,
	pdfstartview=FitH, 
	bookmarksopen=true,
	bookmarksnumbered=true
} \else \hypersetup{dvips} \fi
\allowdisplaybreaks

\numberwithin{equation}{section}

\newtheorem{theorem}{Theorem}[section]
\newtheorem{proposition}[theorem]{Proposition}

\newtheorem{remark}[theorem]{Remark}
\newtheorem{lemma}[theorem]{Lemma}

\newtheorem{assump}[theorem]{Assumptions}

\newcommand {\R}{\mathbb{R}}

\newcommand {\F}{\mathcal{F}}

\newcommand {\p}{\mathbb{P}}

\newcommand {\E}{\mathbb{E}}

\newcommand{\diff}{{\rm d}}

\newcommand{\Gen}{\mathcal{L}}

\newcommand{\Ind}{\mathbf{1}}

\newcommand{\cL}{\mathcal{L}}
\newcommand{\psiq}{\psi_q}
\newcommand{\phiqF}{\varphi_{q,F}}
\newcommand{\phiqFprime}{\varphi_{q,F}^{\prime}}
\newcommand{\phiqFminus}{\phi_{q,F}}

\newcommand{\ABS}[1]{{{\left| #1 \right|}}} 
\newcommand{\BRA}[1]{{{\left\{#1\right\}}}} 
\newcommand{\PAR}[1]{{{\left(#1\right)}}} 
\newcommand{\SBRA}[1]{{{\left[#1\right]}}} 
\renewcommand{\leq}{\leqslant}
\renewcommand{\geq}{\geqslant}

\allowdisplaybreaks

\begin{document}

\title[Optimal withdrawals in a diffusion model]{Optimal withdrawals in a general diffusion model with control rates subject to a state-dependent upper bound}

\author[H. Gu\'erin]{H\'el\`ene Gu\'erin}

\author[D. Mata]{Dante Mata}

\author[J.-F. Renaud]{Jean-Fran\c{c}ois Renaud}
\address{D\'epartement de math\'ematiques, Universit\'e du Qu\'ebec \`a Montr\'eal (UQAM), 201 av.\ Pr\'esident-Kennedy, Montr\'eal (Qu\'ebec) H2X 3Y7, Canada}
\email{guerin.helene@uqam.ca,mata\_lopez.dante@uqam.ca,renaud.jf@uqam.ca}

\author[A. Roch]{Alexandre Roch}
\address{D\'epartement de finance, \'Ecole des sciences de la gestion, Universit\'e du Qu\'ebec \`a Montr\'eal (UQAM), 315 rue Sainte-Catherine Est, Montr\'eal (Qu\'ebec) H2X 3X2, Canada}
\email{roch\_f.alexandre@uqam.ca}


\begin{abstract}

{We consider a classical stochastic control problem in which a diffusion process is controlled by a withdrawal process up to a termination time. The objective is to maximize the expected discounted value of the withdrawals until the first-passage time below level zero.  In this work, we are considering absolutely continuous control strategies in a general diffusion model. Our main contribution is a solution to the control problem under study, which is achieved by using a \textit{probabilistic} guess-and-verify approach. We prove that the optimal strategy belongs to the family of bang-bang strategies, i.e.\ strategies in which, above an optimal barrier level, we withdraw at the highest-allowed rate, while no withdrawals are made below this barrier. Some nontrivial examples are studied numerically. }
\end{abstract}

\maketitle

\section{Introduction}

We consider a classical stochastic control problem in which a nonnegative and increasing control process, called hereafter a \textit{withdrawal process}, is subtracted from a given diffusion process. The objective is to maximize the expected discounted value of the withdrawals made until the termination time, chosen here as the first-passage time (of the controlled process) below level zero. The challenge in this control problem lies in finding an optimal balance between large amounts of withdrawals and longevity of the process. This problem has interpretations in areas such as storage and inventory modelling (\cite{shreve-et-al_1984}), in insurance and financial mathematics (\cite{jeanblanc-shiryaev_1995,paulsen_2007,zhu_2015}), and also in harvesting of populations (\cite{alvarez-shepp_1998}); for example, in the latter two interpretations, a withdrawal strategy is often called a dividend policy and a harvesting plan, respectively. In the existing literature, various models, admissible withdrawal strategies and performance functions have been studied. Without transaction costs, the problem typically results in a singular control problem (\cite{shreve-et-al_1984}), while with fixed transaction costs imposed on each withdrawal it becomes an impulse control problem (\cite{paulsen_2007,bai-paulsen_2010}).

In this work, we consider absolutely continuous controls in a general diffusion model. The novelty of this control problem lies in restricting the control rate process to be less than a level-dependent bound, i.e., a \textit{bound function} applied to the current state of the controlled process, in the spirit of \cite{alvarez-shepp_1998}, \cite{renaud-simard_2021} and \cite{locas-renaud_2024}. In their paper, \cite{alvarez-shepp_1998} have considered two versions of this problem in a logistic diffusion model, respectively with a constant bound function and a linear bound function. Similarly, \cite{renaud-simard_2021} and \cite{locas-renaud_2024} have considered a simple Brownian motion model with a linear bound function and with a nondecreasing and concave bound function, respectively. Previously, \cite{zhu_2015} had considered a general diffusion model but with a constant bound function. Consequently, our paper generalizes these works by tackling the problem in a general diffusion model with a nondecreasing and concave bound function. 

Our main contribution is a solution to the control problem described above, which is achieved by using a \textit{probabilistic} guess-and-verify approach. This method has the advantage of identifying explicitly the dynamics of an optimal strategy and providing an analytical expression for the value function of the control problem. Indeed, we compute first the performance function of an arbitrary refraction strategy, i.e., a bang-bang strategy in which withdrawals are made at the highest possible level-dependent rate above a given threshold and no withdrawals are made below the threshold. Using this semi-explicit analytical form, we then identify the (candidate) optimal threshold level; this is the most challenging step. The final step is the verification procedure, that is a confirmation that this candidate optimal refraction strategy is in fact optimal within the set of all admissible withdrawal strategies.

Another contribution worth mentioning is the analytical study of a functional ($I_F$ below) needed to identify the optimal barrier level. This function is given by the application of the resolvent operator of the underlying diffusion (without killing) to the bound function ($F$ below). In \cite{renaud-simard_2021}, this expectation was computed explicitly, thanks to the combined simplicity of the Brownian model and the linear bound function. In \cite{locas-renaud_2024}, also in a Brownian model but with a general nondecreasing and concave bound function, in which case no explicit expression for this functional is expected, a significant amount of effort was devoted to proving that it is a concave function. Our approach, in a general diffusion model, is based on the notion of concave envelope and on classical comparison results for solutions of ordinary differential equations. As in \cite{locas-renaud_2024}, we prove that this function is concave, which means in particular that the resolvent operator preserves concavity.

Finally, we illustrate our results for some examples of the underlying diffusion model and of the bound function. In particular, when both the drift function and the bound function are affine, we are able to compute explicitly $I_F$ and provide explicit expressions for all the quantities involved. Although it should be noted that the logistic diffusion model, as considered in \cite{alvarez-shepp_1998}, does not fall exactly under the umbrella of our assumptions (the volatility can be zero in that model), we analyze numerically a very similar model.

The rest of the paper is organized as follows. In Section 2, we present the model and the control problem, and then we state our main result (Theorem 2.1) which is the solution to the control problem. Section 3 is devoted to the analytical study of the abovementioned functional ($I_F$) and it is culminating with the proof of its concavity (see Proposition 3.5). In Section 4, we obtain an analytical expression for the performance of an arbitrary barrier strategy and use the concavity of $I_F$ to identify the optimal barrier level, while in Section 5 we prove that this candidate optimal barrier strategy is an optimal strategy for the control problem under study. Finally, in Section 6, we present a few examples of models and bound functions. 

\section{Model, control problem and main result}

Let $(\Omega, \F, (\F_t)_{t\geq 0}, \p)$ be a complete and filtered probability space satisfying the \textit{usual conditions}. We assume that the underlying (uncontrolled) process evolves according to the following stochastic differential equation (SDE):
\begin{equation}\label{eq:main-sde}
\diff Y_t = \mu(Y_t) \diff t + \sigma (Y_t) \diff W_t ,
\end{equation}
where $(W_t)_{t\geq 0}$ denotes a standard Brownian motion adapted to the filtration $(\F_t)_{t\geq 0}$. Assumptions on the functions $\mu$ and $\sigma$ will be provided later; for now, we assume they are regular enough for the existence of a unique strong solution to the above SDE.

As we are interested in absolutely continuous control process/strategies, we identify each strategy with its adapted density process $\ell = (\ell_t)_{t\geq 0}$, which is the rate of withdrawals. Then, the associated controlled process is given by
\begin{equation} \label{eq:X^ell}
	\diff X^\ell_t = \mu(X^\ell_t) \diff t + \sigma(X^\ell_t) \diff W_t - \ell_t \diff t ,
\end{equation}
where $t \mapsto \int_0^t \ell_s \diff s$ represents the amount of withdrawals made up to time $t$. Note that, for any strategy $\ell$, the process $X^\ell$ has continuous sample paths, almost surely.

Let $F \colon [0,\infty) \to \R$ be a nondecreasing and concave function such that $F(0)\geq0$; it is the bound function. As in \cite{locas-renaud_2024}, we say that a control strategy, as defined above, is \textit{admissible} if its density process is such that $\ell_t \in [0, F(X^\ell_t)]$, for all $t\geq 0$. Accordingly, we denote the set of admissible strategies/density rates by $\Pi$.

In what follows, we will write $\E_x$ as the associated expectation operator when $X^\ell$ starts from $X^\ell_0 = x \in \R$, and $\E$ when it starts from $X^\ell_0 = 0$. We will also use the notation $\E_x[V;A]=\E_x[V\Ind_A]$, for a random variable  $V$  and an event $A$.
	
Finally, given a time-preference rate $q>0$, we define the performance function of a strategy $\ell \in \Pi$ by
	\begin{equation}\label{PV_general}
	J(x,\ell) = \E_x\left[ \int_0^{\tau^\ell_0} e^{-qt} \ell_t \diff t \right] , \quad x \in \R ,
	\end{equation}
where $\tau^\ell_0 = \inf\lbrace t\geq 0 \colon X^\ell_t \leq 0 \rbrace$ is the termination time. Note that $J(x,\ell)=0$ for all $x \leq 0$ and for any $\ell \in \Pi$.

The objective of our control problem is to maximise this performance function over all admissible strategies. More precisely, we want to find an expression for the value function
	\[
	V(x) := \sup_{\ell \in \Pi} J(x,\ell) , \quad x \in \R ,
	\]
	and we want to find an optimal control strategy $\ell^\ast$, i.e., an admissible strategy such that $V(x) = J(x,\ell^\ast)$, for all $x \in \R$.

We will reach the above mentioned objectives of computing $V(x)$ and finding an optimal strategy, under the following standing assumptions.
 \begin{assump}\label{assumptions}
     \begin{enumerate}
         \item The drift function $\mu \colon [0,\infty) \to \R$ is concave and it is such that $\mu(0)>0$. In addition, we assume that $\mu'(0+) < q$.
\item The diffusion function $\sigma \colon [0,\infty) \to \R$ is Lipschitz continuous and it is bounded away from zero, i.e., there exists $\epsilon>0$ such that $\sigma (x) \geq \epsilon$ for all $x \in [0,\infty)$.
\item The bound function $F \colon [0,\infty) \to \R$ is a nondecreasing and concave function such that $F(0)\geq0$.
\end{enumerate}
\end{assump}

These assumptions are sufficient for the existence of unique strong solutions $X^\ell$ to the associated stochastic differential equations, as given in~\eqref{eq:X^ell}; see, e.g., Theorem IV.3.1 in \cite{ikeda-watanabe_1989}.

\subsection*{Refraction strategies, an expectation and differential operators}\label{Sec:refract-IF-ODEs}

The experienced reader is expecting a bang-bang strategy to be optimal for this control problem. In this direction, we now define the family of refraction strategies; in \cite{locas-renaud_2024}, these strategies were called \textit{barrier mean-reverting strategies}.

More precisely, for a given $b\geq 0$, we define the refraction strategy with barrier $b$ by the following density process:
	\begin{equation}\label{eq:l^b}
	\ell^b_t = F(X^b_t) \Ind_{\lbrace X^b_t \geq b\rbrace} , \quad t\geq 0,
	\end{equation}
where $X^b=(X^b_t)_{t \geq 0}$ is given by the solution of
\begin{equation}\label{eq:X^b}
	\diff X^b_t = \left( \mu(X^b_t) - F(X^b_t) \Ind_{\lbrace X^b_t \geq b\rbrace} \right) \diff t + \sigma(X^b_t) \diff W_t.
\end{equation}

This is indeed a specific example of the SDE given in~\eqref{eq:X^ell}. The term \textit{refraction} comes from the fact that the drift function is either $\mu(\cdot)$ or $\mu(\cdot)-F(\cdot)$, whether the process is below or above $b$.

We denote the performance function of the refraction strategy at level $b$ by
\begin{equation}\label{Barrier:PV}
J_b(x) = \E_x\SBRA{\int_0^{\tau^b_0} e^{-qt} F(X^b_t) \Ind_{\lbrace X^b_t \geq b\rbrace} \diff t } , \quad x \in \R,
\end{equation}
where $\tau^b_0 = \inf\lbrace t\geq 0 \colon X^b_t \leq 0 \rbrace$.

\begin{remark}
Note that, for any $b\geq 0$, the process $X^b$ is the unique strong solution of~\eqref{eq:X^b} since its drift function is a piecewise Lipschitz continuous function; see, e.g., \cite{leobacher-szolgyenyi_2016}. In particular, a refraction strategy at level $b$ is admissible.
\end{remark}

Note that, if $b=0$, then $X^0$ is the controlled process associated with the strategy consisting in withdrawing at the maximal rate until the termination time. Note that, up to its termination time $\tau_0^0$, this process has the same distribution as $X=(X_t)_{t \geq 0}$ given by
\begin{equation}\label{eq:X}
\diff X_t = (\mu(X_t) - F(X_t)) \diff t + \sigma(X_t) \diff W_t .
\end{equation}

Now, let us define one of the key functions in our study:
	\begin{equation}\label{IF:definition}
	I_F(x) = \E_x \SBRA{\int_0^{\infty} e^{-qt} F(X_t) \diff t} , \quad x\in \R.
	\end{equation}

In order for $I_F$ to be well defined, we must specify the behaviour of $X$ in $(-\infty, 0)$:

\begin{remark}[Extension of $\mu,\sigma,F$ on $\R$]
\label{remark:extensions}
For $x \in (-\infty,0)$, set
\[
	\mu(x) = \mu(0) + \mu'(0+)x, \quad F(x) = F(0) + F'(0+)x, \quad \sigma(x) = \sigma(0) ,
\]
where $\mu'(0+)$ and $F'(0+)$ are right derivatives at $0$. Note that $\mu$ and $F$ are now concave functions on $\R$, while $\sigma$ is Lipschitz continuous on $\R$.
\end{remark}

The function $I_F$ is of paramount importance in our approach. As alluded to above, the main challenge in our approach is to obtain some analytical properties for $I_F$. This task is undertaken in Section~\ref{Sec:I_F:Properties}.

Finally, denote the infinitesimal generator associated with $X$ by
\[
\Gen_F = \frac{\sigma^2(x)}{2} \frac{\diff^2}{\diff x^2} + (\mu(x) - F(x)) \frac{\diff}{\diff x}.
\]
Similarly, we write $\Gen$ for the infinitesimal generator associated with $Y$ or, equivalently, corresponding to $X$ with $F \equiv 0$. As it is well known (see, e.g., \cite{borodin2002handbook}), we can denote by $\psiq$ the positive and increasing fundamental solution of
\[
(\Gen - q)u(x) = 0 , \quad x \in (0,\infty) ,
\]
with initial conditions $u(0)=0$ and $u^\prime(0+)=1$. Similarly, let $\phiqF$ be the positive and decreasing fundamental solution of 
\[
(\Gen_F - q)u(x) = 0 , \quad x \in (0,\infty) ,
\]
with boundary conditions $u(0) = 1$ and $\lim_{x\to\infty} u(x)=0$. It is also well known that these functions appear in first-passage identities related to the corresponding process. In particular, for $b\geq 0$ and $0 \leq x \leq b$, we have
	\begin{equation}\label{Y:exit:time}
		\E_x\left[ e^{-q \kappa_b} ; \kappa_b < \kappa_0 \right] = \frac{\psi_q(x)}{\psi_q(b)}, 
	\end{equation}
where $\kappa_b = \inf\lbrace t \geq 0 \colon Y_t \geq b \rbrace$ and $\kappa_0 = \inf\lbrace t \geq 0 \colon Y_t \leq 0 \rbrace$, while for $x \geq b$, we have
	\begin{equation}\label{X:exit:time}
		\E_x\left[ e^{-q {\tau_b} } \right] = \frac{\phiqF(x)}{\phiqF(b)} ,
	\end{equation}
where $\tau_b = \inf\lbrace t \geq 0 \colon X_t \leq b \rbrace$.

Here are useful properties that we will use repeatedly in the rest of the paper; see \cite[Lemma 2.4]{ekstrom-lindensjo_2023}.
\begin{lemma}\label{lemma:concave-convex}
Under Assumptions~\ref{assumptions}, $\psiq$ is concave-convex on $(0,\infty)$ with a unique inflection point $\hat{b}>0$ and $\phiqF$ is strictly convex on $(0,\infty)$.
\end{lemma}

We now state our main result.
\begin{theorem}\label{Thm:b:Optimal}
Under Assumptions~\ref{assumptions}, there exists an optimal refraction level $b^\ast \geq 0$ such that the refraction strategy $\ell^{b^\ast}$ is optimal for the control problem. Consequently, we have
\[
V(x) = \E_x\SBRA{\int_0^{\tau^\ast_0} e^{-qt} F(X^{b^\ast}_t) \Ind_{\lbrace X^{b^\ast}_t \geq b^\ast\rbrace} \diff t } , \quad x\geq 0,
\]
where $X^{b^\ast}$ satisfies the SDE \eqref{eq:X^b} and $\tau^\ast_0 = \inf\lbrace t\geq 0 \colon X^{b^\ast}_t \leq 0 \rbrace$. Also, $V$ is a twice continuously differentiable function characterised by the following dichotomy: 
\begin{itemize}
    \item If $I_F^\prime(0+) - I_F(0) \phiqFprime(0+) > 1$, then $b^\ast>0$ is a root of
\[
\frac{\psiq(b)}{\psiq^\prime(b)} = I_F(b) + \left(1- I_F^\prime(b) \right) \frac{\phiqF(b)}{\phiqFprime(b)},
\]
and $V$ is a concave function given by:
\[
V(x) =
\begin{cases}
\frac{\psiq(x)}{\psiq^\prime(b^\ast)} & \text{for $x \leq b^\ast$,}\\
I_F(x) - \phiqF(x) \left( \frac{1-I_F^\prime(b^\ast)}{\phiqFprime(b^\ast)} \right) & \text{for $x \geq b^\ast$}.
\end{cases}
\]
\item If $I_F^\prime(0+) - I_F(0) \phiqFprime(0+) \leq 1$, then $b^\ast = 0$ and 
\[
V(x) = I_F(x) - I_F(0) \phiqF(x), \quad x \geq 0 .
\]
\end{itemize}
\end{theorem}

The remainder of the paper is devoted to proving Theorem~\ref{Thm:b:Optimal}. In particular, in Section~\ref{Sect:refraction}, we will compute the performance function of an arbitrary refraction strategy. 

\section{Analytical properties of $I_F$}\label{Sec:I_F:Properties}

As mentioned above, our proof of Theorem~\ref{Thm:b:Optimal} follows a probabilistic guess-and-verify approach based on the computation of the performance function of an arbitrary refraction strategy. The function $I_F$ defined in~\eqref{IF:definition}, and based on the process $X$ defined in~\eqref{eq:X}, will appear in that expression and its analytical properties, especially concavity, are needed to prove the existence of an optimal threshold.

Before moving forward, note that $I_F$ takes finite values:
\begin{proposition}\label{Prop:Defined}
For $x \in \R$, we have $|I_F(x)| < \infty$.
\end{proposition}

A technical proof of this last proposition is provided in Appendix~\ref{App:Defined}. 

\begin{proposition}\label{rk:IF-mu-F-affine}
Under Assumptions~\ref{assumptions}, we have the following:
\begin{enumerate}
\item If $\mu$ and $F$ are affine functions, then $I_F$ is affine and given by
    \begin{equation}\label{eq:IF-mu-F-affine}
I_F(x)=\frac{F(0)}{q}+\frac{F'(0)}{q(q-\mu'(0)+F'(0))}\PAR{qx+\mu(0)-F(0)}, \quad x\in\R.
\end{equation} 
\item In general, for $x \leq 0$, we have
\begin{equation}\label{IF-under-zero}
I_F(x) = \alpha_0 + \beta_0 x + \left( I_F(0)-\alpha_0 \right) \phiqFminus(x),
\end{equation}
where $\beta_0:=\frac{F^\prime(0+)}{q-\mu^\prime(0+)+F^\prime(0+) }$, $\alpha_0:=\frac{\mu(0)}{q} \beta_0 + \frac{F(0)}{q} (1-\beta_0)$, and $\phiqFminus$ is the unique increasing solution of
\[
(\Gen_F - q) u(x) = 0, \quad x \in (-\infty,0) ,
\]
with boundary conditions $u(0) = 1$ and $\lim_{x \to -\infty}u(x) = 0$.
\end{enumerate}
\end{proposition}
\begin{proof}
If $\mu$ and $F$ are affine functions, then we can write
\begin{align*}
\mu(x)=\mu(0)+\mu'(0)x\quad\text{and}\quad F(x)=F(0)+F'(0)x,\quad x\in\R.
\end{align*}
The corresponding dynamics for $X$ is
\[ 
\diff X_t=\SBRA{(\mu(0)-F(0))+(\mu'(0)-F'(0))X_t}\diff t+\sigma(X_t)\diff W_t ,
\]
and thus we have for $t\geq 0$
\[
\E_x[X_t]=x+(\mu(0)-F(0))t+(\mu'(0)-F'(0))\int_0^t\E_x[X_s]\diff s.
\]
When $\mu'(0)\neq F'(0)$, we have
\[
\E_x[X_t]=\PAR{x+\frac{\mu(0)-F(0)}{\mu'(0)-F'(0)}} \mathrm{e}^{\PAR{\mu'(0)-F'(0)}t}-\frac{\mu(0)-F(0)}{\mu'(0)-F'(0)}
 \]
and, when $\mu'(0)=F'(0)$, we have $\E_x[X_t]=x+(\mu(0)-F(0))t$.

Also, if $\mu$ and $F$ are affine, then
\[
I_F(x)=\frac{F(0)}{q}+F'(0)\int_0^{\infty}\mathrm{e}^{-qt}\E_x[X_t]\diff t
\]
and, since $\mu'(0)<q$ by Assumptions~\ref{assumptions}, we deduce Equation~\eqref{eq:IF-mu-F-affine}. This proves (1).

In the general case, i.e., under Assumptions~\ref{assumptions} together with the extensions on $(-\infty,0)$ in Remark \ref{remark:extensions}, let us consider the Ornstein-Uhlenbeck process $\tilde{X}$ given by
\begin{equation}\label{eq:Xtilde}
\diff \tilde{X}_t = \left[ \mu(0)- F(0) + \left(\mu^\prime(0+) - F^\prime(0+) \right) \tilde{X}_t \right] \diff t + \sigma(0) \diff W_t .
\end{equation}
Note that if $x\leq 0$ then $\tilde X$ and $X$ are equal in law (with respect to $\p_x$) up to time $\rho_0:=\inf\BRA{t \geq 0: \tilde{X}_t \geq 0}$. In addition, as for~\eqref{X:exit:time}, if $x\leq 0$, then we have
\[
\E_x\left[ e^{-q {\rho_0} } \right] = \phiqFminus(x) ,
\]
where $\phiqFminus$ is the unique increasing solution of
\[
(\Gen_F - q) u(x) = 0, \quad x \in (-\infty,0) ,
\]
with boundary conditions $u(0) = 1$ and $\lim_{x \to -\infty}u(x) = 0$. 

Successive applications of the strong Markov property give, for $x \leq 0$,
\begin{equation}\label{eq:comment-helene}
I_F(x) = \tilde{I}_F(x) - \tilde{I}_F(0) \phiqFminus(x) + I_F(0) \phiqFminus(x) ,
\end{equation}
where $\tilde{I}_F(x) := \E_x \SBRA{\int_0^{\infty} e^{-qt} \left(F(0) + F^\prime(0+) \tilde{X}_t \right) \diff t}$.

Consequently, from (1), we have
\[
\tilde{I}_F(x) = \alpha_0 + \beta_0 x ,
\]
where
\begin{equation}\label{alpha:0}
\alpha_0 = \frac{1}{q-\mu^\prime(0+)+F^\prime(0+)} \left[ F(0) + \frac{\mu(0) F^\prime(0+)-\mu^\prime(0+) F(0)}{q} \right] = \frac{\mu(0)}{q} \beta_0 + \frac{F(0)}{q} (1-\beta_0) .
\end{equation}
This proves (2).
\end{proof}

\begin{proposition}\label{prop:I_F of class C^2}
The function $I_F$ is twice continuously differentiable on $\R$. Moreover, $I_F$ satisfies
\[
(\Gen_F - q)I_F(x) = - F(x), \quad x \in \R.
\]
\end{proposition}

\begin{proof}
For $x \geq 0$, we have
\[
\E_x\left[ e^{-q {\tau_0} } \right] = \frac{\phiqF(x)}{\phiqF(0)} = \phiqF(x),
\]
where $\tau_0 = \inf\lbrace t \geq 0 \colon X_t \leq 0 \rbrace$. This is a special case of Equation~\eqref{X:exit:time}. Consequently, an application of the strong Markov property yields, for $x \geq 0$,
\[
I_F(x) = J_0(x) + I_F(0) \phiqF(x) ,
\]
where
\begin{equation}\label{Resolvent:J0}
J_0(x) = \E_x \SBRA{\int_0^{\tau_0} e^{-qt} F(X_t) \diff t}
\end{equation}
is the performance function of the refraction strategy with threshold level $b=0$.

On one hand, we have that $\phiqF \in \mathcal{C}^2(0,\infty)$. On the other hand, since the functional $J_0$ is obtained by an application of the resolvent operator (killed at zero) of $X$ to the bound function $F$ (see, e.g., Equations (18)--(20) in \cite{alvarez2003}), we have that $J_0$ is a classical solution to the following second-order initial value problem:
\begin{equation}\label{resolvent:ODE2}
\begin{cases}
(\Gen_F - q) J_0(x) = - F(x), \quad x \in (0,\infty),\\
J_0(0) = 0 .
\end{cases}
\end{equation}
It follows that $J_0 \in \mathcal{C}^2(0,\infty)$. As a consequence, we have also that $I_F \in \mathcal C^2(0,\infty)$ and
\[
(\Gen_F - q) I_F(x) = -F(x), \quad x \in (0,\infty).
\]

For $x \in (-\infty,0)$, we have
\[
\Gen_F = \frac{\sigma^2(0)}{2} \frac{\diff^2}{\diff x^2} + \left[\mu(0)- F(0) + \left(\mu^\prime(0+) - F^\prime(0+) \right) x \right] \frac{\diff}{\diff x} .
\]
From the expression of $I_F$ on $(-\infty,0)$ given in~\eqref{IF-under-zero} of Proposition~\ref{rk:IF-mu-F-affine}, we easily deduce that
\[
(\Gen_F - q) I_F(x) = -F(x), \quad x \in (-\infty,0).
\]

So far, we have shown that $I_F \in \mathcal C^2( \R\setminus\{0\} )$. So, to conclude, it suffices to show that $I_F'(0-) = I_F'(0+)$ and $I_F''(0-) = I_F''(0+)$. Let us proceed by contradiction.

First, suppose $I_F'(0+) < I_F'(0-)$ and let $d \in (I_F'(0+), I_F'(0-) )$. For an arbitrary $M >0$, define the function 
\[\pi_M(x) := I_F(0) + dx - \frac{1}{2}M x^2, \quad x\in \R.\]
Note that $\pi_M(0) = I_F(0)$ and, since $d  \in (  I_F'(0+), I_F'(0-) )$, by continuity, there exists $\epsilon_M > 0$ such that $\pi_M(x) \geq I_F(x)$ for all $x \in [-\epsilon_M,\epsilon_M]$.

Let $h>0$ and set $\tau_M = \inf\BRA{ t\geq 0: X_t \notin [-\epsilon_M,\epsilon_M] } \wedge h$. Using the strong Markov property, we have
\begin{align*}
    \frac{1}{h} \pi_M(0) &= \frac{1}{h} I_F(0) \\
    &= \frac{1}{h} \E \SBRA{ \int_0^{\tau_M} e^{-qt} F(X_t) \diff t + e^{-q\tau_M} I_F(X_{\tau_M}) }\\
    &\leq \frac{1}{h} \E \SBRA{ \int_0^{\tau_M} e^{-qt} F(X_t) \diff t + e^{-q\tau_M} \pi_M(X_{\tau_M}) }.
\end{align*}
Using It\^{o}'s Formula, we can further expand the term $e^{-q\tau_M} \pi_M(X_{\tau_M})$ and hence we obtain
\begin{align*}
    0 & \leq \frac{1}{h}\E\SBRA{ \int_0^{\tau_M} e^{-qt}\left( F(X_t) + ( \mu(X_t) - F(X_t) )( d - M X_t ) - \frac{1}{2}\sigma^2(X_t) M - q \pi_M(X_t) \right) \diff t }.
\end{align*}
By taking the limit as $h \downarrow 0$, and re-arranging terms, we obtain
\[
q \pi_M(0) - (\mu(0) - F(0)) d + \frac{1}{2}\sigma^2(0)M - F(0) \leq 0 .
\]

However, this last inequality is not verified for $M > \frac{2}{\sigma^2(0)}\PAR{F(0) - q \pi_M(0) + (\mu(0)-F(0))d}$, leading to a contradiction. So, we must have $I_F'(0+) \geq I_F'(0-)$. In a very similar way, we can prove that the inequality $I_F'(0+) > I_F'(0-)$ also leads to a contradiction. The details are left to the reader. In conclusion, we  have that $I_F'(0-) = I_F'(0+)$ and hence $I_F \in C^1(\R)$.

Recall that $I_F$ is twice continuously differentiable on $\R\setminus\{0\}$. It is also such that
\[
\PAR{\cL_F-q} I_F=-F , \quad x \in (-\infty,0) \cup (0,\infty) .
\]
Taking the limits when $x\downarrow 0$ and $x\uparrow 0$, and using the continuity at zero of the functions, we obtain
\[
\begin{split}
(\mu(0) - F(0)) I_F'(0-) + \frac12 \sigma^2(0) I_F''(0-) - q I_F(0-) + F(0) &= 0, \\
(\mu(0)-F(0)) I_F'(0+) + \frac12 \sigma^2(0) I_F''(0+) - q I_F(0+) + F(0) &= 0.
\end{split}
\]
We have shown above that $I_F(0+)=I_F(0)$ and $I_F'(0-)=I_F'(0+)$ and, from Equation~\eqref{IF-under-zero} of Proposition~\ref{rk:IF-mu-F-affine}, we easily deduce that $I_F(0-)=I_F(0)$. Therefore, from the above two equalities, we now have that $I_F''(0-)=I_F''(0+)$ and hence the result follows.

\end{proof}

\begin{proposition}\label{Prop:Increasing:Derivative}
For $x \in \R$, we have $I_F'(x) \in [0,1]$.
\end{proposition}

\begin{proof}
First, the monotonicity of $I_F$ follows directly from the fact $F$ is increasing together with the comparison theorem for diffusions (see, e.g., \cite[Theorem 1.1, Chap. VI]{ikeda-watanabe_1989}). In other words, we have $I_F'(x) \geq 0$ for all $x \in \R$.

Next, fix $x \in \R$ and $T>0$. By an application of Dynkin's formula with the (deterministic) stopping time $T$, we can write
\[
\E\left[ e^{-qT} X^x_T \right] = x + \E\left[ \int_0^T e^{-qt} \left( \mu(X_t^x) - F(X_t^x) - q X_t^x \right) \diff t \right] ,
\]
in which the  notation $X^x$  is used (only) in this proof to emphasize that $X^x_0=x$. Taking the limit as $T \to \infty$, using Lemma~\ref{rem:transversal} and re-arranging terms, we get
\begin{equation}\label{IF-second-representation}
I_F(x) = x + \E\left[ \int_0^{\infty} e^{-qt} \left( \mu(X_t^x) - q X_t^x \right) \diff t \right], \quad x \in \R.
\end{equation}
For $h>0$, since $I_F$ is nondecreasing, we have $I_F(x+h) - I_F(x) \geq 0$. In addition, we have
\[
I_F(x+h) - I_F(x) = h + \E\left[ \int_0^{\infty} e^{-qt} \left( \mu(X_t^{x+h}) - \mu(X_t^x) - q( X_t^{x+h} - X_t^x) \right) \diff t \right] \leq h ,
\]
where we used that, for $y_1 \leq y_2$, we have $\mu(y_2) - \mu(y_1) \leq q (y_2-y_1)$ and $X_t^{x+h} \geq X_t^x$ for all $t\geq 0$. Thus, we have obtained
\[
\frac{I_F(x+h) - I_F(x)}{h} \leq 1
\]
and as $I_F'$ exists, the result follows by taking the limit as $h \to 0$.
\end{proof}

We now prove that the growth of $I_F$ is at most linear, with an explicit estimation which will be key to prove the concavity.

\begin{lemma}\label{lem_bound}
For $\xi \in \R$, define $\beta_\xi = \frac{F'(\xi)}{q-\mu'(\xi) + F'(\xi)}$ and $\alpha_\xi = \frac{\beta_\xi}{q} \left[-q \xi + \mu(\xi) + (q-\mu'(\xi))\frac{F(\xi)}{F'(\xi)} \right]$. We have
\begin{equation}\label{envelope_bound}
I_F(x) \leq \alpha_\xi + \beta_\xi x , \quad x \in \R .
\end{equation}
In addition, for $x\leq 0$, we have
\begin{equation}\label{envelope_neg_bound}
I_F(0) + \beta_0 x \leq I_F(x) \leq \alpha_0 + \beta_0 x.
\end{equation}
\end{lemma}

\begin{proof}
First, note that, for $\xi < 0$, we have $\beta_\xi=\beta_0$ and $\alpha_\xi=\alpha_0$ (see definition of $\alpha_0,\beta_0$ in Proposition~\ref{rk:IF-mu-F-affine}).

Let $\xi\in\R$ and $x\in \R$. For simplicity, in this proof, we write $\beta$ instead of $\beta_\xi$. Using~\eqref{IF-second-representation}, we can write
\begin{equation}\label{envelope_1}
\beta x = I_F(x) + \E\left[ \int_0^{\infty} e^{-qt} \left( \beta \left[ q X_t^x - \mu(X_t^x)  + \left(\frac{\beta - 1}{\beta}\right) F(X_t^x)\right] \right) \diff t \right] . 	
\end{equation}

First, note that $\beta>0$ and $ \frac{\beta - 1}{\beta} = \frac{\mu'(\xi) - q}{F'(\xi)} \leq 0$. Second, since $F$ is concave on $\R$, we have $F(X_t^x) \leq F(\xi) + F'(\xi)( X_t^x - \xi )$. Consequently,

\begin{align}
\beta x - I_F(x) &\geq \E\left[ \int_0^{\infty} e^{-qt} \left( \beta \left[ q X_t^x - \mu(X_t^x)   - (q - \mu'(\xi))\left(\frac{F(\xi)}{F'(\xi)} + X_t^x - \xi\right) \right] \right) \diff t \right] \notag\\
		&=\E\left[ \int_0^{\infty} e^{-qt} \left( \beta \left[  - \mu(X_t^x)  - (q - \mu'(\xi))\frac{F(\xi)}{F'(\xi)} + q  \xi + \mu'(\xi)(X_t^x - \xi) \right] \right) \diff t \right] \notag\\
		&\geq \int_0^{\infty} e^{-qt} \left( \beta \left[  q  \xi - \mu(\xi)  - (q - \mu'(\xi))\frac{F(\xi)}{F'(\xi)}  \right] \right) \diff t \label{envelope_2}
\end{align}
where, in the last inequality, we used the fact that $\mu(X_t^x) \leq \mu(\xi) + \mu'(\xi)(X_t^x - \xi)$, since $\mu$ is a concave function on $\R$. It follows immediately from \eqref{envelope_2} that $\beta x-I_F(x)\geq -\alpha_\xi$.

Finally, for the proof of \eqref{envelope_neg_bound}, recall from Equation~\eqref{IF-under-zero} that, for $x \leq 0$,
\[
I_F(x) = \alpha_0 + \beta_0 x + \left( I_F(0)-\alpha_0 \right) \phiqFminus(x) = I_F(0) + \beta_0 x + \left( \alpha_0-I_F(0) \right) \left(1-\phiqFminus(x) \right) .
\]
In addition, it follows from \eqref{envelope_bound} by taking $\xi = 0$ that $I_F(0) \leq \alpha_0$. We deduce easily that, for $x \leq 0$,
\[
I_F(0) + \beta_0 x \leq I_F(x) \leq \alpha_0 + \beta_0 x .
\]
\end{proof}

\begin{remark}\label{rk:monotony-beta}
Under Assumptions~\ref{assumptions}, we have that $\xi \mapsto \beta_{\xi}$ is non-increasing.
\end{remark}

Let us now investigate the concavity of $I_F$. First, note that concavity on $(-\infty,0)$ easily follows from Equation~\eqref{IF-under-zero}. Indeed, for $x \leq 0$,
\[
I_F(x) =  I_F(0) + \beta_0 x + \left( \alpha_0-I_F(0) \right) \left(1-\phiqFminus(x) \right) .
\]
As $\phiqFminus$ is convex (see Lemma~\ref{lemma:concave-convex}) and $\alpha_0 \geq I_F(0)$, it follows easily that $I_F$ is concave on $(-\infty,0)$. Furthermore, as $I_F \in \mathcal{C}^2$, we can deduce that $I''_F(0)\leq 0$.

\begin{proposition}\label{Prop:IF:Concave}
The function $I_F$ is concave on $\R$.
\end{proposition}

\begin{proof}
By Proposition~\ref{rk:IF-mu-F-affine}, if $\mu$ and $F$ are affine, then $I_F$ is affine and therefore concave. We now assume that $\mu$ nor $F$ is affine.

Let $U$ be the concave envelope of $I_F$ on $\R$ given by
\[
U(z) = \inf \BRA{ \Phi(z) : \Phi \text{ is concave and } \Phi \geq I_F }, \quad z \in \R.
\]
It is itself a concave function. We refer the reader to Appendix \ref{App:Envelope} for more details on concave envelopes.

To obtain the concavity of $I_F$, we prove that $I_F=U$, and indeed it suffices to prove that $I_F \geq U$. Let us proceed by contradiction. Therefore, suppose there exists $x_0 >0$ such that $U(x_0) > I_F(x_0)$. As $I_F \in \mathcal{C}^2(\R)$ (see Proposition~\ref{prop:I_F of class C^2}), there exist $-\infty \leq \hat x<\hat y \leq \infty$ such that $x_0 \in (\hat x, \hat y)$, $U(z) > I_F(z)$ for all $z \in (\hat x,\hat y)$, and $U(\hat x) = I_F(x_0)$ if $\hat x > -\infty$ (resp. $U(\hat y) = I_F(x_0)$ if $\hat y < \infty$). We will study separately the following cases: 
\begin{enumerate}
    \item \label{case1} $-\infty < \hat x < \hat y < +\infty$;
    \item \label{case2} $\hat x > -\infty$ and $\hat y = +\infty$;
    \item \label{case3} $\hat x = -\infty$ and $\hat y < +\infty$;
    \item \label{case4} $\hat x = -\infty$ and $\hat y =  +\infty$.
\end{enumerate}

First, we will prove, in cases \eqref{case1}--\eqref{case3}, that
\[
    ( q - \Gen_F )U(z) \leq F(z), \quad \text{for }z\in (\hat x, \hat y).
\]

For case~\eqref{case1},  from Lemma~\ref{lem:affine} and Lemma~\ref{lemma:interval}, we have that $U$ is affine on $(\hat x, \hat y)$ and is such that
\[
U(\hat x) = I_F(\hat x), \; U(\hat y) = I_F(\hat y), \; U'=I_F'(\hat x) = I_F'(\hat y), \; I_F''(\hat x) \leq 0, \; I_F''(\hat y) \leq 0.
\]
Thus, using also that $(q - \Gen_F)I_F = F$ on $\R$ from Proposition~\ref{prop:I_F of class C^2}, we obtain
 \begin{equation}\label{ineq1}
 \begin{split}
     (q - \Gen_F)U(\hat{x}) - F(\hat{x}) &= q I_F(\hat{x}) - (\mu(\hat{x}) - F(\hat{x})) I_F'(\hat{x}) - F(\hat{x})\\
 &= (q - \Gen_F)I_F(\hat x) - F(\hat{x}) + \frac{\sigma^2(\hat{x})}{2} I_F''(\hat{x})\\
     &= \frac{\sigma^2(\hat{x})}{2} I_F''(\hat{x})\leq 0 .
 \end{split}
 \end{equation}
Similarly, we have $(q - \Gen_F)U(\hat{y}) - F(\hat{y}) \leq 0$.

For $z\in (\hat{x},\hat{y})$, we have $z = \lambda \hat{x} + (1-\lambda)\hat{y}$ for some $\lambda \in (0,1)$. Using the concavity of both $\mu$ and $F$ and using Proposition~\ref{Prop:Increasing:Derivative}, we have
\begin{align*}
	(q - \Gen_F)U(z) - F(z) &= q( \lambda U(\hat{x}) + (1-\lambda)U(\hat{y}) ) - I_F'(\hat{x})\mu(z) + F(z)(I_F'(\hat{x}) - 1)\\
	&\leq q( \lambda U(\hat{x}) + (1-\lambda)U(\hat{y}) ) - I_F'(\hat{x}) (\lambda \mu(\hat{x}) + (1-\lambda)\mu(\hat{y}))\\
	&\qquad  + (I_F'(\hat{x}) - 1) (\lambda F(\hat{x}) + (1-\lambda)F(\hat{y}))\\
	& \leq \lambda ((q - \Gen_F)U(\hat{x}) - F(\hat{x})) + (1-\lambda)((q - \Gen_F)U(\hat{y}) - F(\hat{y}))\\
	&\leq 0,
\end{align*}
where the last inequality is a direct consequence of inequality~\eqref{ineq1} and its equivalent with $\hat x$ replaced by $\hat y$. It follows that 
\begin{equation}\label{Concave_1}
(q - \Gen_F)U(z) \leq F(z), \quad z \in [\hat{x},\hat{y}].
\end{equation}

Now, for case~\eqref{case2},  from Lemma~\ref{lem:affine} and Lemma~\ref{lemma:interval}, we know that $U$ is affine on $(\hat{x},\infty)$ and is such that
\begin{equation}\label{envelope_infty}
U(\hat x) = I_F(\hat x), \; U'(\hat x) = I_F'(\hat x), \; I_F''(\hat x) \leq 0 .
\end{equation}
In particular, we have $U(z)=I_F(\hat{x})+I_F^\prime(\hat{x})(z-\hat{x})$ for $z \in (\hat{x},\infty)$. Then, for $z\in (\hat x,\infty)$, we have
\begin{align*}
    ( q - \Gen_F )U(z) - F(z) &= q I_F(\hat{x}) + I_F'(\hat{x}) ( q (z - \hat{x})  - \mu(z) + F(z)) - F(z)\\
    &= I_F'(\hat{x})( \mu(\hat{x}) - F(\hat{x} )) + \frac{\sigma^2(\hat{x})}{2}I_F''(\hat{x}) + F(\hat{x})\\
    & \qquad + I_F'(\hat{x}) ( q (z - \hat{x})  - \mu(z) + F(z)) - F(z) ,
\end{align*}
where in the second equality we used that $ ( q - \Gen_F )I_F(\hat{x}) = F(\hat{x}) $ in order to substitute for $q I_F(\hat{x})$. After re-arranging terms, we get
\begin{equation}\label{envelope_4}
( q - \Gen_F )U(z) - F(z) = \frac{\sigma^2(\hat{x})}{2}I_F''(\hat{x}) + I_F'(\hat{x}) \left( q (z-\hat{x}) - (\mu(z) - \mu(\hat{x})) + \frac{I_F'(\hat{x})-1}{I_F'(\hat{x})}( F(z) - F(\hat{x}) ) \right).
\end{equation}
Now, let us obtain a suitable bound for the term $\frac{I_F'(\hat{x})-1}{I_F'(\hat{x})}$. To this end, we notice by Lemma~\ref{lem_bound} and by the definition of the concave envelope that, for all $\xi\in\R$,
\[
U(x)\leq  \alpha_\xi+\beta_\xi x,\quad x\in\R.
\]
Since $U$ is affine with slope $I_F'(\hat x)$ on $(\hat{x},\infty)$, we deduce $I_F'(\hat x)\leq \beta_\xi$ for any $\xi\in\R$. 

Using the definition of $\beta_\xi$ and $-q+\mu'(\xi)<0$, we have that
\begin{equation}\label{slope_bound}
\frac{I_F'(\hat x) -1 }{I_F'(\hat x)} \leq \frac{-q + \mu'(\xi)}{F'(\xi)}.
\end{equation}
Equipped with this upper bound, we return to the analysis of \eqref{envelope_4}. Applying the inequality \eqref{slope_bound} with $\xi = z$ in \eqref{envelope_4} yields 
\begin{align*}
( q - \Gen_F )U(z) - F(z) &\leq \frac{\sigma^2(\hat{x})}{2}I_F''(\hat{x}) + I_F'(\hat x) (z - \hat{x}) \left( q - \frac{\mu(z) - \mu(\hat{x})}{z-\hat{x}} -\left( \frac{q - \mu'(z)}{ F'(x) } \right) \frac{F(z) - F(\hat{x})}{z-\hat{x}} \right)\notag \\
    & \leq \frac{\sigma^2(\hat{x})}{2}I_F''(\hat{x}) + I_F'(\hat x) (z - \hat{x}) \left( q - \mu'(z) - \left( \frac{q - \mu'(z)}{ F'(z) } \right) F'(z) \right)\\
    &= \frac{\sigma^2(\hat{x})}{2}I_F''(\hat{x}) \leq 0,
\end{align*}
where we used that $I_F$ and $F$ are increasing, that $\mu$ and $F$ are concave, and \eqref{envelope_infty}. This allows us to conclude that 
\begin{equation}\label{envelope_6}
    ( q - \Gen_F )U(z) \leq F(z), \quad z\in (\hat x, +\infty).
\end{equation}
In case~\eqref{case3},  from Lemma~\ref{lem:affine} and Lemma~\ref{lemma:interval}, we have that 
$U(z)=I_F(\hat{y})+I_F^\prime(\hat{y})(z-\hat{y})$ for $z \in (-\infty,\hat{y})$, with
\begin{equation}\label{envelope_infty2}
U(\hat y) = I_F(\hat y), \; U'(\hat y) = I_F'(\hat y), \; I_F''(\hat y) \leq 0 .
\end{equation}
Using the bounds \eqref{envelope_neg_bound} of $I_F$ and the definition of the concave envelope, we deduce that
\[
I_F(0)+\beta_0 z\leq U(z)\leq \alpha_0+\beta_0 z,\quad z\leq \min(0,\hat y).
\]
This implies that $I_F'(\hat y)=\beta_0$. Furthermore, by monotonicity of $\xi\mapsto\beta_\xi$ (see Remark~\ref{rk:monotony-beta}), for any $\xi\in\R$, $\frac{I_F'(\hat y)-1}{I_F'(\hat y)}=\frac{-q+\mu'(0+)}{F'(0+)}\geq \frac{-q+\mu'(\xi)}{F'(\xi)}$. Similarly to case~\eqref{case2}, we deduce
\begin{equation}\label{envelope_8}
( q - \Gen_F )U(z) \leq F(z), \quad z\in (-\infty, \hat y).
\end{equation}

At last in case~\eqref{case4}, by Lemma~\ref{lem:affine} we know that $U$ is affine on $\R$: $U(z)=U(0)+U'(0)z$. Similarly to case~\eqref{case2}, we deduce that $U'(0)\leq \beta_\xi$ for any $\xi\in\R$, and similarly to case~\eqref{case3}, we obtain $U'(0)=\beta_0$. From monotonicity of $\xi\mapsto \beta_\xi$ in Remark~\ref{rk:monotony-beta}, we deduce that case~\eqref{case4} can only happen when $\beta_\xi$ is constant. Under Assumptions~\ref{assumptions} this is only true when $F'$ and $\mu'$ are constant, which is impossible by hypothesis.

We are now ready to obtain a contradiction in each case \eqref{case1}--\eqref{case3}. 
Define the exit time $\widehat T = \inf\lbrace t \geq 0 \colon X_t \notin (\hat{x},\hat{y}) \rbrace$. Note that, 
on $\BRA{\widehat{T}<+\infty}$ we have $X_{\widehat{T}} \in \{\hat{x},\hat{y}\}$ in case~\eqref{case1}; $X_{\widehat{T}} = \hat{x}$ in case~\eqref{case2}; and $X_{\widehat{T}} = \hat{y}$ in case~\eqref{case3}. 
As $x_0 \in (\hat{x},\hat{y})$, by the strong Markov property, we have
\begin{align}
	I_F(x_0) &= \E_{x_0} \left[ \int_0^{\widehat T} e^{-qt} F(X_t) \diff t + e^{-q \widehat{T}} I_F(X_{\widehat T});\widehat{T}<+\infty \right]+\E_{x_0} \left[ \int_0^{+\infty} e^{-qt} F(X_t) \diff t ;\widehat{T}=+\infty \right] \\
	&\geq \E_{x_0} \left[ \int_0^{\widehat T} e^{-qt} (q-\Gen_F)U(X_t) \diff t + e^{-q \widehat{T}} I_F(X_{\widehat T})  \Ind_{\widehat{T}<+\infty}\right] \\
	&= \E_{x_0} \left[ \int_0^{\widehat T} e^{-qt} (q-\Gen_F)U(X_t) \diff t + e^{-q \widehat{T}} U(X_{\widehat T}) \Ind_{\widehat{T}<+\infty}\right] \label{envelope_ito}.
\end{align}
Note that in \eqref{envelope_ito} we used the fact that, on $\BRA{\widehat{T}<\infty}$,  $I_F(\hat{x})=U(\hat{x})$ and/or $I_F(\hat{y})=U(\hat{y})$. 

In order to conclude, we analyse the term on the right side of \eqref{envelope_ito}. Let 
$\tau_n := \widehat{T}\wedge n$. Then, using It\^{o}'s Formula we get
\begin{align*}
    e^{-q \tau_n}U(X_{\tau_n}) - U(x_0) &= \int_0^{\tau_n} e^{-qt} \PAR{ \Gen_F - q } U(X_t) \diff t + \int_0^{\tau_n} \sigma(X_t) U'(X_t) \diff W_t.
\end{align*}
By taking expectations on both sides, we notice that the stochastic integral term vanishes, 
hence we obtain
\begin{equation}\label{envelope_ito_3}
\E_{x_0} \SBRA{ e^{-q \tau_n}U(X_{\tau_n}) } - U(x_0) = \E_{x_0}\SBRA{ \int_0^{\tau_n} e^{-qt} \PAR{ \Gen_F - q } U(X_t) \diff t }.
\end{equation}
Observe that we can expand the term $\E_{x_0} \SBRA{ e^{-q \tau_n}U(X_{\tau_n}) }$ as
\[
\E_{x_0} \SBRA{ e^{-q \tau_n}U(X_{\tau_n}) } = \E_{x_0} \SBRA{ e^{-q \tau_n}U(X_{\tau_n}) \Ind_{\lbrace \widehat{T} < +\infty \rbrace} } + \E_{x_0} \SBRA{ e^{-q n}U(X_{\tau_n}) \Ind_{\lbrace \widehat{T} = +\infty \rbrace} },
\]
where it follows that $ \E_{x_0} \SBRA{ e^{-q n}U(X_{\tau_n}) \Ind_{\lbrace \widehat{T} = +\infty \rbrace} } \to 0 $ as $n\to \infty$ due to the fact that $U$ is affine and Lemma~\ref{rem:transversal}, and $ \E_{x_0} \SBRA{ e^{-q \tau_n}U(X_{\tau_n}) \Ind_{\lbrace \widehat{T} < +\infty \rbrace} } \to \E_{x_0} \SBRA{ e^{-q \widehat{T}}U(X_{\widehat{T}}) \Ind_{\lbrace \widehat{T} < +\infty \rbrace} } $ as $n\to \infty$. Hence, by taking the limit in \eqref{envelope_ito_3} as $n\to \infty$ we get
\[
\E_{x_0} \SBRA{ e^{-q \widehat{T}}U(X_{\widehat{T}}) \Ind_{\lbrace \widehat{T} < +\infty \rbrace} } - U(x_0) = \E_{x_0}\SBRA{ \int_0^{\widehat{T}} e^{-qt} \PAR{ \Gen_F - q } U(X_t) \diff t },
\]
and, by re-arranging terms, we obtain 
\[
\E_{x_0} \left[ \int_0^{\widehat T} e^{-qt} (q-\Gen_F)U(X_t) \diff t + e^{-q \widehat{T}} U(X_{\widehat T}) \Ind_{\widehat{T}<+\infty}\right] = U(x_0),
\]
where it follows easily that $I_F(x_0) \geq U(x_0)$.

In conclusion, we have obtained a contradiction to the statement: there exists $x_0 >0$ such that $U(x_0) > I_F(x_0)$. Thus $I_F=U$ and $I_F$ is a concave function.
 \end{proof}

\section{A Candidate optimal strategy}\label{Sect:refraction}

In this section, we study the family of refraction strategies which are defined by Equations~\eqref{eq:l^b} and~\eqref{eq:X^b}. Despite the fact that we consider a general diffusion model, the rest of this section follows closely the general methodology already used in \cite{locas-renaud_2024} for a Brownian motion with drift.\\

We provide an intuitive justification as to why we are interested in proving the optimality of refraction strategies. To this end, we observe that the Hamilton-Jacobi-Bellman (HJB) equation associated with our control problem is given by
\begin{equation}\label{HJB}
    \begin{cases}
        ( \Gen - q ) u(x) + \sup_{\mathrm{c} \in [0,F(x)]} \mathrm{c} (1 - u'(x)) = 0, \quad x \in (0,\infty),\\
        u(0) = 0.
    \end{cases}
\end{equation}
From the HJB equation, one can deduce that, if $u$ is a classical solution of \eqref{HJB}, then we have
\begin{equation}
    \begin{cases}
        ( \Gen - q ) u(x) = 0, & \text{if } u'(x) \geq 1,\\
        ( \Gen_F - q ) u(x) = -F(x), & \text{if } u'(x) < 1.
    \end{cases}
\end{equation}
These variational inequalities have the following interpretation:
\begin{itemize}
    \item If $u'(x) \geq 1$, then the supremum in \eqref{HJB} is attained for $\mathrm{c} = 0$ and no withdrawals are made;
    \item If $u'(x) < 1$, then the supremum in \eqref{HJB} is attained for $\mathrm{c} = F(x)$ and withdrawals are made at the maximal possible rate.
\end{itemize}
It follows that our goal is to find a classical solution \eqref{HJB} and characterise the regions 
\begin{equation}\label{regions}
\mathcal{D}_0 := \BRA{x \in [0,\infty): u'(x) \geq 1 } \text{ and } \mathcal{D}_F := \BRA{x \in [0,\infty): u'(x) < 1}.
\end{equation}
In the context of refraction strategies, given a threshold $b\geq 0$, we notice that the region where no withdrawals are made is the interval $[0,b]$; whereas the region where withdrawals are made at the maximal possible rate is the interval $(b,\infty)$. Hence, we would like to prove the existence of a threshold $b^\ast$ such that its associated performance function is a classical solution of \eqref{HJB} and satisfies $\mathcal{D}_0 = [0,b^\ast]$ and $\mathcal{D}_F = (b^\ast,\infty)$.

Recall the performance function of a refraction strategy at level $b\geq 0$, as defined in \eqref{Barrier:PV}:
\[
J_b(x) = \E_x\SBRA{\int_0^{\tau^b_0} e^{-qt} F(X^b_t) \Ind_{\lbrace X^b_t \geq b\rbrace} \diff t } , \quad x \in \R.
\] 
First, we will derive an expression for $J_b$ using the exit identities \eqref{Y:exit:time} and \eqref{X:exit:time}. In order to do so, let us fix $b>0$. Define $\kappa_b = \inf\lbrace t \geq 0 \colon Y_t \geq b \rbrace$ and $\kappa_0 = \inf\lbrace t \geq 0 \colon Y_t \leq 0 \rbrace$. Then, for $x \in [0,b]$, using the fact that $X^b$ and $Y$ are equal until they exit the interval $(0,b)$ and using the strong Markov property of $Y$, we can write
	\begin{equation}\label{J_b_1}
	J_b(x) = \E_x\left[ e^{-q \kappa_b} ; \kappa_b < \kappa_0 \right] J_b(b) =\frac{\psiq(x)}{\psiq(b)} J_b(b) .
	\end{equation}

Similarly, for $x \in (b,\infty)$, using repeatedly the strong Markov property, we can write
\begin{align}
J_b(x) &= \E_x \left[ \int_0^{{\tau_b}} e^{-qt} F({X}_t) \diff t \right] + \E_x\left[ e^{-q {\tau_b} } \right] J_b(b)\notag\\
&=  \E_x \left[ \int_0^{\infty} e^{-qt} F({X}_t) \diff t \right] - \E_x \left[ \int_{{\tau}_b}^{\infty} e^{-qt} F({X}_t) \diff t \right] + \E_x\left[ e^{-q {\tau}_b } \right] J_b(b)\notag\\
&= \E_x \left[ \int_0^{\infty} e^{-qt} F({X}_t) \diff t \right] - \E_x\left[ e^{-q {\tau}_b } \right] \E_b \left[ \int_0^{\infty} e^{-qt} F({X}_t) \diff t \right] + \E_x\left[ e^{-q {\tau}_b } \right] J_b(b)\notag\\
&= I_F(x) + \frac{\phiqF(x)}{\phiqF(b)} ( J_b(b) - I_F(b) ) .\label{J_b_2}
\end{align}

Using the same approximation procedure as in \cite{locas-renaud_2024}, one can prove that
\[
J_b(b) = \psiq(b) \left(\frac{I_F'(b)\phiqF(b) - \phiqFprime(b)I_F(b)}{\psiq^\prime(b) \phiqF(b) - \phiqFprime(b) \psiq(b)} \right).
\]
The details are left to the reader.

Substituting this value of $J_b(b)$ back into the expressions obtained above, we get
\begin{equation}\label{eq:J>b-2}
J_b(x) =
\begin{cases}
\psiq(x) \left( \frac{ I_F'(b) \phiqF(b) - I_F(b) \phiqFprime(b)}{\psiq^\prime(b)\phiqF(b) - \psiq(b) \phiqFprime(b)} \right) & \text{for $0 \leq x < b$,}\\
I_F(x) + \phiqF(x) \left( \frac{I_F'(b) \psiq(b) - I_F(b) \psiq^\prime(b)}{\psiq^\prime(b) \phiqF(b) - \psiq(b) \phiqFprime(b)} \right) & \text{for $x \geq b$.}
\end{cases}
\end{equation}

It is interesting to note that, for $b>0$, since we have $J'_b(b+) = J'_b(b-)$, then $J_b \in \mathcal{C}^1(0,\infty)$.

On the other hand, for $b=0$, recall from the proof of Proposition~\ref{prop:I_F of class C^2} that $J_0$ is a classical solution of 
\[
\begin{cases}
(\Gen_F - q)u(x) = -F(x) , \quad x \in (0,\infty), \\
u(0) = 0.
\end{cases}
\]

The next step in our approach consists in finding the \textit{best} refraction strategy. In other words, we want to characterize the optimal refraction level. Mathematically speaking, we want to find a value $b^\ast \geq 0$ such that $J_{b^\ast} \geq J_b$, for any $b\geq 0$. In what follows, we will distinguish between the cases $b^\ast=0$ and $b^\ast > 0$.

For reasons that will soon be clear, if $I_F^\prime(0) - I_F(0) \phiqFprime(0+) \leq 1$, then we set $b^\ast=0$. If $I_F^\prime(0) - I_F(0) \phiqFprime(0+) > 1$, then our goal is to find a value $b^\ast>0$ such that $\mathcal{D}_0 = [0,b^\ast]$ and $\mathcal{D}_F = (b^\ast,\infty)$. From the discussion at the beginning of this section, we expect that $b^\ast$ will be such that $J'_{b^\ast}(b^\ast)=1$, which after some algebraic manipulations amounts to finding the root of
\begin{equation}\label{b_optimal}
\frac{\varphi_{q,F}(b)}{\varphi_{q,F}^\prime(b)} - \frac{\psi_q(b)}{\psi_{q}^\prime(b)} = I_F'(b)\frac{\varphi_{q,F}(b)}{\varphi_{q,F}^\prime(b)} - I_F(b). 
\end{equation}
Of course, we need to prove that a unique solution to this last equation exists. This is provided by the next proposition. It is in its proof that the results of Section~\ref{Sec:I_F:Properties}, especially Proposition~\ref{Prop:Increasing:Derivative} and Proposition~\ref{Prop:IF:Concave}, are needed.

\begin{proposition}\label{prop_exist_b}
If $I_F^\prime(0) - I_F(0) \phiqFprime(0+) > 1$, then there exists a solution to~\eqref{b_optimal} in the interval $(0,\hat{b}]$.
\end{proposition}

\begin{proof}
We follow the proof of Proposition 4.1 in \cite{locas-renaud_2024}. Set $g(b) := \frac{\phiqF(b)}{\phiqFprime(b)} - \frac{\psiq(b)}{\psiq^\prime(b)}$ and $h(b):= I_F'(b)\frac{\phiqF(b)}{\phiqFprime(b)} - I_F(b)$. Since $h(0)\leq g(0)$ by assumption, and both $g$ and $h$ are continuous, by the Intermediate Value Theorem, it suffices to prove that $g(\hat{b}) \leq h(\hat{b})$.

First, since $\hat{b}$ is the unique inflection point of $\psi_q$, we have $\frac{\psi_q(\hat{b})}{\psi_{q}^\prime(\hat{b})} = \frac{\mu(\hat{b})}{q}$, hence $g(\hat{b}) = \frac{\varphi_{q,F}(\hat{b})}{\varphi_{q,F}^\prime(\hat{b})} - \frac{\mu(\hat{b})}{q}$. In addition, thanks to the convexity of $\varphi_{q,F}$  we have $g(\hat{b}) \leq \frac{-F(\hat{b})}{q}.$
 \\
	Now, we focus on the computation of $h(\hat{b})$. Recall that $I_F$ satisfies the non-homogeneous ODE 
 \[(\Gen_F - q) I_F(b) + F(b) = 0.\] 
 In particular, for $\hat{b}$, and due to the concavity of $I_F$ we have
	\begin{align*}
		I_F(\hat{b}) &= \frac{\sigma^2(\hat{b})}{2 q}I_F''(\hat{b}) + \frac{\mu(\hat{b})}{q}I_F'(\hat{b}) + \frac{F(\hat{b})}{q}(1 - I_F'(\hat{b}) )\\
		&\leq \frac{\mu(\hat{b})}{q}I_F'(\hat{b}) + \frac{F(\hat{b})}{q}(1 - I_F'(\hat{b}) ).
	\end{align*}
	It follows that
	\begin{align*}
		h(\hat{b}) &\geq I_F'(\hat{b})\frac{\varphi_{q,F}(\hat{b})}{\varphi_{q,F}^\prime(\hat{b})} - \left( \frac{\mu(\hat{b})}{q}I_F'(\hat{b}) + \frac{F(\hat{b})}{q}(1 - I_F'(\hat{b}) ) \right)\\
		& = I_F'(\hat{b}) \left( \frac{\varphi_{q,F}(\hat{b})}{\varphi_{q,F}^\prime(\hat{b})} - \frac{\mu(\hat{b})}{q} \right) - \frac{F(\hat{b})}{q}(1 - I_F'(\hat{b}) )\\
		& = I_F'(\hat{b}) g(\hat{b}) - \frac{F(\hat{b})}{q}(1 - I_F'(\hat{b}) ) \\
		&\geq I_F'(\hat{b}) g(\hat{b}) + g(\hat{b}) (1 - I_F'(\hat{b}) ),
	\end{align*}
since $I_F'(\hat{b})\in[0,1]$. The result follows.
\end{proof}

\begin{remark}\label{Remark:Condition}
Note that the condition $I_F^\prime(0) - I_F(0) \phiqFprime(0+) > 1$ is equivalent to $J_0'(0+) > 1$.
\end{remark}

If $I_F^\prime(0) - I_F(0) \phiqFprime(0+) > 1$, then, as discussed above, by taking a solution $b^*$ to~\eqref{b_optimal}  we have $J'_{b^*}(b^*)=1$. Thus, it follows from~\eqref{J_b_1} and~\eqref{J_b_2} that $\frac{J_{b^\ast}(b^\ast)}{\psi_q(b^\ast)}=\frac{1}{\psi^{q'}(b^\ast)}$ and $\frac{J_{b^\ast}(b^\ast) - I_F(b^\ast)}{\varphi_{q,F}(b^\ast)}=\frac{ 1 - I_F'(b^\ast) }{ \varphi_{q,F}^\prime(b^\ast) }$. So, when $b^*>0$, we can write
	\begin{equation}\label{J_optimal}
	J_{b^\ast}(x) = \begin{cases} \frac{\psi_q(x)}{\psi_{q}^\prime(b^\ast)}, & 0 \leq x \leq b^\ast ,\\
	I_F(x) + \varphi_{q,F}(x)\left( \frac{ 1 - I_F'(b^\ast) }{ \varphi_{q,F}^\prime(b^\ast) } \right), & x \geq b^\ast . \end{cases}
	\end{equation}

\section{Verification of optimality}

In this section, we complete the proof of Theorem~\ref{Thm:b:Optimal}. First, let us recall that, if $I_F^\prime(0) - I_F(0) \phiqFprime(0+) \leq 1$, then we set $b^\ast = 0$, and if $I_F^\prime(0) - I_F(0) \phiqFprime(0+) > 1$, then $b^\ast>0$ is a solution to~\eqref{b_optimal}.

\begin{proposition}\label{Prop:Optim:Concave}
If $b^\ast=0$, then $J'_{0}(x) \in [0,1]$ for all $x>0$. If $b^\ast>0$, then $J_{b^\ast}$ is concave. In both cases, we have that $J_{b^\ast} \in \mathcal{C}^2(0,\infty)$ and it is a classical solution to the HJB equation given in~\eqref{HJB}.
\end{proposition}

\begin{proof}
First, note that if $I_F^\prime(0) - I_F(0) \phiqFprime(0+) \leq 1$, then $I_F(0) \geq \frac{I_F^\prime(0)-1}{\phiqFprime(0+)}$. Consequently, together with the fact that $J_0(x) = I_F(x) - I_F(0) \phiqF(x)$ for $x\in [0,\infty)$, that $I_F$ is concave and that $\phiqF$ is strictly decreasing, we can write, for $x>0$,
\begin{equation}\label{J0_bound_1}
J_0'(x) = I_F'(x) - I_F(0)\phiqFprime(x) \leq I_F'(0) - \frac{I_F'(0)-1}{\phiqFprime(0+)} \phiqFprime(0+) = 1 .
\end{equation}

As $J_0$ is nondecreasing, we conclude that $J_0'(x) \in [0,1]$ for all $x \in (0,\infty)$.

We have already mentioned that
\[
(\Gen_F - q)J_0(x) = - F(x), \quad x \in (0,\infty).
\]
Consequently, we have
\[
(\Gen - q) J_0(x) + \sup_{l \in [0,F(x)]} l (1 - J'_0(x)) = (\Gen_F - q) J_0(x) + F(x) J'_0(x) + F(x) (1 - J'_0(x)) = 0 ,
\]
for all $x>0$. In other words, $J_0$ is a classical solution to the HJB Equation in~\eqref{HJB}.

Now, we consider the case $b^\ast > 0$, i.e., we assume $I_F^\prime(0) - I_F(0) \phiqFprime(0+) > 1$. Using the representation of $J_{b^*}$ given in~\eqref{J_optimal}, along with the fact that $b^\ast \leq \hat{b}$, that $\varphi_{q,F}$ is decreasing, and that $J_{b^*} \in \mathcal{C}^1 (0,\infty)$, it follows easily that $J_{b^\ast}$ is concave on $\R^+$. It remains to verify that $J_{b^*} \in \mathcal C^2 (0,\infty)$. To this end, we have to prove that $ J''_{b^\ast}(b^\ast -) = J''_{b^\ast}(b^\ast +)$.

On one hand, we have
	\[
	J''_{b^\ast}(b^\ast -) = \frac{\psi_q^{\prime\prime}(b^\ast)}{\psi_{q}^\prime(b^\ast)} .
	\]
Note that, by definition, $\psi_q$ is such that $(\Gen - q) \psi_q(x) = 0$, for $x \in (0,\infty)$, which is equivalent to
	\begin{equation}\label{c2_aux1}
	\frac{\psi_q^{\prime\prime}(b^\ast)}{\psi_{q}^\prime(b^\ast)} = \frac{2}{\sigma^2(b^\ast)}\left(  q \frac{\psi_q(b^\ast)}{\psi_{q}^\prime(b^\ast)} - \mu(b^\ast)  \right).
	\end{equation}
On the other hand, we have
	\begin{align*}
		J''_{b^\ast}(b^\ast +) &= I_F''(b^\ast) + \varphi^{\prime\prime}_{q,F}(b^\ast)\left( \frac{ 1 - I_F'(b^\ast) }{ \varphi_{q,F}^\prime(b^\ast) } \right).
	\end{align*}
Again, by definition, $\varphi_{q,F}$ is such that $(\Gen_F - q) \varphi_{q,F}(x) = 0$, for $x \in (0,\infty)$, which is equivalent to
	\begin{equation}\label{c2_aux2}
		\frac{\varphi_{q,F}^{\prime\prime}(b^\ast)}{\varphi_{q,F}^\prime(b^\ast)} = \frac{2}{\sigma^2(b^\ast)}\left(  q \frac{\varphi_{q,F}(b^\ast)}{\varphi_{q,F}^\prime(b^\ast)} - (\mu(b^\ast) - F(b^\ast) )  \right) .
	\end{equation}
Similarly, for $I_F$, we can write
	\begin{equation}\label{c2_aux3}
		I_F''(b^\ast) + \frac{2}{\sigma^2(b^\ast)}(\mu(b^\ast) - F(b^\ast))I_F'(b^\ast) = \frac{2}{\sigma^2(b^\ast)}(q I_F(b^\ast) - F(b^\ast) ) .
	\end{equation}
Hence, by substituting~\eqref{c2_aux2} and~\eqref{c2_aux3} in the expression for $J''_{b^\ast}(b^\ast +$, we can write
	\begin{align*}
		J''_{b^\ast}(b^\ast +) &= I_F''(b^\ast) +  \frac{2}{\sigma^2(b^\ast)}\left(  q \frac{\varphi_{q,F}(b^\ast)}{\varphi_{q,F}^\prime(b^\ast)} - (\mu(b^\ast) - F(b^\ast) )  \right) ( 1 - I_F'(b^\ast) )\\
		&= \frac{2q}{\sigma^2(b^\ast)}\left( I_F(b^\ast) - \frac{\varphi_{q,F}(b^\ast)}{\varphi_{q,F}^\prime(b^\ast)} I_F'(b^\ast) \right) + \frac{2q}{\sigma^2(b^\ast)} \frac{\varphi_{q,F}(b^\ast)}{\varphi_{q,F}^\prime(b^\ast)} - \frac{2}{\sigma^2(b^\ast)} \mu(b^\ast) \\
		&= \frac{2q}{\sigma^2(b^\ast)} \left( \frac{\psi_q(b^\ast)}{\psi_{q}^\prime(b^\ast)} - \frac{\varphi_{q,F}(b^\ast)}{\varphi_{q,F}^\prime(b^\ast)} \right) + \frac{2q}{\sigma^2(b^\ast)} \frac{\varphi_{q,F}(b^\ast)}{\varphi_{q,F}^\prime(b^\ast)} - \frac{2}{\sigma^2(b^\ast)} \mu(b^\ast)\\
		& = \frac{2}{\sigma^2(b^\ast)}\left( q\frac{\psi_q(b^\ast)}{\psi_{q}^\prime(b^\ast)}  - \mu(b^\ast) \right) = J_{b^*}(b^*-) ,
	\end{align*}
	where in the penultimate equality we used Proposition \ref{prop_exist_b}.

Now, from the smooth-fit condition~\eqref{b_optimal} and the concavity of $J_{b^\ast}$, we have
\begin{equation}\label{eq:J-prime}
\begin{cases}
J'_{b^\ast}(x) > 1 & \text{if $x \in (0,b^\ast)$,}\\
J'_{b^\ast}(x) \leq 1 & \text{if $x \in (b^\ast,\infty)$.}
\end{cases}
\end{equation}
As a consequence, since we also have
\begin{equation}\label{eq:J-HJB}
\begin{cases}
(\Gen - q) J_{b^\ast}(x)=0 & \text{if $x \in (0,b^\ast)$,}\\
(\Gen - q) J_{b^\ast}(x) = - F(x) (1-J_{b^\ast}^\prime(x)) & \text{if $x \in (b^\ast,\infty)$,} 
\end{cases}
\end{equation}
it follows that $J_{b^\ast}$ is a classical solution of the HJB equation given in~\eqref{HJB}.
\end{proof}

\subsection{Proof of Theorem~\ref{Thm:b:Optimal}}
We have already proven in Proposition \ref{prop_exist_b} that $b^\ast$ is well-defined, and in Proposition \ref{Prop:Optim:Concave} that $J_{b^\ast}$ is a classical solution of the associated HJB equation. All is left to prove is that $J_{b^\ast}(x)\geq V(x)$ for all $x \geq 0$.

First note that from \eqref{eq:J-prime} and \eqref{eq:J-HJB}, for $x>0$ and $u \in [0,F(x)]$, we have
\begin{equation}\label{J_b_generator}
(\Gen - q) J_{b^\ast}(x) + u\left(1-J_{b^\ast}^\prime(x) \right) \leq 0 .
\end{equation}
Let $\ell$ be an arbitrary admissible strategy and $X^\ell$ the associated controlled process with ruin time $\tau_0^\ell$. Hence, as $J_{b^\ast} \in \mathcal{C}^2 (0,\infty)$, we can apply It\^{o}'s Formula to obtain
\begin{align*}
e^{-q{t\wedge\tau_0^\ell}} & J_{b^\ast}(X^\ell_{t\wedge\tau_0^\ell}) - J_{b^\ast}(X^\ell_0) \\
&= M_t^\ell + \int_0^{t\wedge\tau_0^\ell}e^{-qs} \PAR{ \frac{1}{2}\sigma^2(X^\ell_s) J_{b^\ast}''(X^\ell_s) + (\mu(X^\ell_s) - \ell_s) J_{b^\ast}'(X^\ell_s) - q J_{b^\ast}(X^\ell_s)}\diff s\\
&= M_t^\ell + \int_0^{t\wedge\tau_0^\ell} e^{-qs} \PAR{ (\Gen-q) J_{b^\ast}(X^\ell_s) + (1-J_{b^\ast}^\prime(X^\ell_s)) \ell_s^\ell - \ell_s^\ell} \diff s
\end{align*}
where $M_t^\ell := \int_0^{t\wedge\tau_0^\ell} e^{-qs}\sigma(X^\ell_s) J_{b^\ast}'(X^\ell_s) \diff W_s $ is a local martingale.

Let $(\rho_n)_{n\geq 1}$ be a localising sequence and set $T_n := t\wedge\tau_0^\ell\wedge\rho_n$. Since $\ell$ is an admissible strategy, we can use~\eqref{J_b_generator} with $u=\ell_s$ and take expectations (with initial value $x \geq 0$) to get
\[
\E_x \SBRA{ e^{-q{T_n}}J_{b^\ast}(X^\ell_{T_n}) } - J_{b^\ast}(x) \leq - \E_x \SBRA{ \int_0^{T_n} e^{-qs} \ell_s \diff s } .
\]

Since $J_{b^\ast} \geq 0$, we further have
    \[
    J_{b^\ast}(x) \geq \E_x \SBRA{ \int_0^{T_n} e^{-qs} \ell_s \diff s } ,
    \]
from which we take limits when $t \to \infty$ and $n\to \infty$, yielding
    \[
    J_{b^\ast}(x) \geq J(x,\ell) = \E_x \SBRA{ \int_0^{\tau_0^\ell} e^{-qs} \ell_s \diff s } , \quad \text{for all $x \geq 0$.}
    \]
    Since this last inequality holds for any admissible strategy $\ell$, the proof is complete.

\section{Examples}\label{sec:examples}

In this section, we first present some examples where the functions $I_F,\varphi_{q,F}$ and $\psi_q$ are explicit, followed by a numerical analysis of a logistic-type model.

\subsection{When $\mu$ and $F$ are affine functions}\label{sec:affine-examples}

Let $\mu$ and $F$ be affine functions given by
\[
\mu(x)=\mu_0+\mu_1 x \quad \text{and} \quad F(x)=F_0+F_1 x .
\]
Of course, all functions must satisfy our standing assumptions.

From Proposition~\ref{rk:IF-mu-F-affine}, we have
\[
I_F(x)=\frac{1}{q}F_0+\frac{F_1}{q(q-c_1)}\PAR{qx+c_0} ,
\]
where $c_i=\mu_i-F_i$, for $i=0,1$.

From Proposition~\ref{prop_exist_b}, when $ \frac{F_1}{(q-c_1)} -I_F(0)\varphi_{q,F}^\prime(0)>1$, there exists a unique positive solution $b^*$ to
\begin{equation}\label{eq:def-bstar}
\frac{\varphi_{q,F}(b)}{\varphi_{q,F}^\prime(b)} - \frac{\psi_q(b)}{\psi_{q}^\prime(b)} = \frac{F_1}{q-c_1}\frac{\varphi_{q,F}(b)}{\varphi_{q,F}^\prime(b)} - I_F(b) ,
\end{equation}
and then the value function is $V = J_{b^\ast}$. If the above condition is not satisfied, then we have $V=J_0$.

Before looking at specific choices of $\sigma(x)$ for which $\varphi_{q,F}$ and $\psi_q$ are well-known \textit{special functions}, let us introduce a few intermediate results. More details are provided in Appendix~\ref{App:proof-example} and we refer the reader to \cite{abramowitz-stegun_1964} for further properties of the special functions used in this section. 

Let us first introduce the Kummer confluent hypergeometric function $_1F_1(a,b;z)$ (also denoted $M(a,b;z)$ in the literature) and  the confluent hypergeometric function of the second kind $U(a,b;z)$, which are linearly independent solutions of 
\[
zu''+(b-z)u'-au=0.
\]
The function $U(a,b;z)$ is an explicit linear combination of $_1F_1(a,b;z)$ and $z^{1-b}\,_1F_1(1+a-b,2-b;z)$. When $a,b>0$, $_1F_1(a,b;z)$ can be represented as an integral
\begin{equation}\label{eq:1F1-integral}
_1F_1(a,b;z)=\frac{\Gamma(b)}{\Gamma(b-a)\Gamma(a)}\int_0^1\mathrm{e}^{zt}t^{a-1}(1-t)^{b-a-1}\diff t,
\end{equation}
and when $a>0$, $U(a,b;z)$ has also an integral representation
\begin{equation}\label{eq:U-integral}
U(a,b;z)=\frac{1}{\Gamma(a)}\int_0^\infty {\rm e}^{-zt}t^{a-1}(1+t)^{b-a-1}\diff t.
\end{equation}
The first derivatives of $_1F_1(a,b;\cdot)$ and $U(a,b;\cdot)$ are such that
\begin{align*}
\PAR{_1F_1}'(a,b;z)&=\frac{a}{b}\,_1F_1(a+1,b+1;z)\\
U'(a,b;z)&=-aU(a+1,b+1;z).
\end{align*}

We will also use the parabolic cylinder function $D_{-\lambda}$, which is a solution to $u''-(\lambda -1/2+z^2/4)u=0$. It has the following integral representation when $\lambda>0$:
\begin{equation}\label{eq:D-integral}
D_{-\lambda}(z)=\frac{1}{\Gamma(\lambda)}\mathrm{e}^{-z^2/4}\int_0^\infty \mathrm{e}^{-zt-t^2/2}t^{\lambda-1}\diff t,
\end{equation}
and its first derivative satisfies $D_{-\lambda}'(z)=-\frac{1}{2}zD_{-\lambda}(z)-\lambda D_{-\lambda-1}(z)$.

\subsubsection{When $\sigma$ is a constant}

Assume $\sigma(x)=\sigma > 0$. In this case, $X$ and $Y$ are Ornstein-Uhlenbeck processes; note that this specific case has been already studied in \cite[Section 4.1]{locas-renaud_2024}. The corresponding operator is given by
\[
\PAR{\cL_F -q}u=\frac{\sigma^2}{2}u''+\PAR{c_0+c_1x}u'-qu.
\]

We extend now the result of \cite{locas-renaud_2024,renaud-simard_2021} for any $c_1\in\mathbb{R}$ and complete their theoretical analysis with additional computations. For Ornstein-Uhlenbeck processes, introducing $a_F=\frac{q}{|c_1|}+\frac{\mathrm{sign}(c_1)}{2}$, we have 
\begin{align}\label{eq:phi-sigma-constant}
\varphi_{q,F}(x)&=
\begin{cases}
\mathrm{e}^{-\frac{x(2c_0+c_1x)}{2\sigma^2}}\frac{D_{-a_F-\frac{1}{2}}\PAR{\mathrm{sign}(c_1)\sqrt{\frac{2}{|c_1|}}\frac{c_0+c_1x}{\sigma}}}{D_{-a_F-\frac{1}{2}}\PAR{\mathrm{sign}(c_1)\sqrt{\frac{2}{|c_1|}}\frac{c_0}{\sigma}}}\quad& \text{when }c_1\neq 0,\\[0.4cm]
\mathrm{e}^{-\frac{\sqrt{c_0^2+2q\sigma^2}+c_0}{\sigma^2}x}\quad &\text{when }c_1=0.
\end{cases}
\end{align}

We easily see that $\varphi_{q,F}(0)=1$, and by \eqref{eq:D-integral} that $\varphi_{q,F}$ is decreasing to $0$. Moreover, we have
\begin{align}\label{eq:psi-sigma-constant}
\psi_q(x)&=
\begin{cases}
\kappa\mathrm{e}^{-\frac{x(2\mu_0+\mu_1x)}{2\sigma^2}} \frac{y_1\PAR{a_0;\sqrt{\frac{2}{|\mu_1|}}\frac{\mu_0+\mu_1x}{\sigma}}}{y_1\PAR{a_0;\sqrt{\frac{2}{|\mu_1|}}\frac{\mu_0}{\sigma}}}-\kappa \varphi_0(x)\quad & \text{when }\mu_1\neq 0,
\\[0.2cm]
\kappa\mathrm{e}^{\frac{\sqrt{\mu_0^2+2q\sigma^2}-\mu_0}{\sigma^2}x}-\kappa\varphi^q_0(x)\quad &\text{when }\mu_1=0,\\
\end{cases}
\end{align}
where  $y_1(a;z)=\,_1F_1\PAR{\frac{1}{2}a+\frac{1}{4},\frac{1}{2};\frac{z^2}{2}}\mathrm{e}^{-z^2/4}$, $a_0:=a_F$ and $\varphi_0^q:=\varphi_{q,F}$ for $F\equiv 0$, and $\kappa$ is such that $ \PAR{\psi_q}'(0)=1$. 
Note that we can compute easily  $\PAR{\varphi_{q,F}}'(0)$, and $\kappa$ has an explicit expression in terms of the model parameters. When $\mu_1,c_1\neq 0$, we have 
\[
\varphi_{q,F}'(0)=-\frac{c_0}{\sigma^2}+\frac{\sqrt{2|c_1|}}{\sigma}\frac{D'_{-a_F-\frac{1}{2}}\PAR{{\mathrm{sign}(c_1)\sqrt{2}c_0}/{\sqrt{|c_1|}\sigma}}}{D_{-a_F-\frac{1}{2}}\PAR{{\mathrm{sign}(c_1)\sqrt{2}c_0}/{\sqrt{|c_1|}\sigma}}}
\]
and 
\[
\kappa^{-1}=\frac{\sqrt{2|\mu_1|}}{\sigma}\SBRA{\mathrm{sign}(\mu_1)\frac{y'_1\PAR{a_0;{\sqrt{2}\mu_0}/{\sqrt{|\mu_1|}\sigma}}}{y_1\PAR{a_0;{\sqrt{2}\mu_0}/{\sqrt{|\mu_1|}\sigma}}}-\frac{D'_{-a_0-\frac{1}{2}}\PAR{{\mathrm{sign}(\mu_1)\sqrt{2}\mu_0}/{\sqrt{|\mu_1|}\sigma}}}{D_{-a_0-\frac{1}{2}}\PAR{{\mathrm{sign}(\mu_1)\sqrt{2}\mu_0}/{\sqrt{|\mu_1|}\sigma}}}}.
\]

\subsubsection{When $\sigma$ is an affine function}

Assume $\sigma(x)=\sigma_0+\sigma_1x$, with $\sigma_0, \sigma_1> 0$. 
For example, if $\mu_0=F_0=\sigma_0=0$, then $X$ and $Y$ are geometric Brownian motions.  The operator is given by
\[
\PAR{\cL_F -q}u=\frac{1}{2}\PAR{\sigma_0+\sigma_1 x}^2u''+\PAR{c_0+c_1x}u'-qu.
\]

Let $\Delta_F>0$ with $\Delta_F^2=(\sigma_1^2-2c_1)^2+8q\sigma_1^2$, $a_F=\PAR{\Delta_F-\sigma_1^2+2c_1}/\PAR{2\sigma_1^2}$ and $b_F=1+{\Delta_F}/{\sigma_1^2}$.  Then, 
\begin{align}\label{eq:phi-sigma-affine}
\varphi_{q,F}(x)&=\PAR{1+ \frac{\sigma_1}{\sigma_0}x}^{-a_F}\times\frac{_1F_1\PAR{a_F,b_F;\frac{2\PAR{c_0\sigma_1-c_1\sigma_0}}{\sigma_1^2\PAR{\sigma_0+\sigma_1x} }}}{_1F_1\PAR{a_F,b_F;\frac{2\PAR{c_0\sigma_1-c_1\sigma_0}}{\sigma_1^2\sigma_0 }}}.
\end{align}
We easily observe that $\varphi_{q,F}(0)=1$ and $\lim_{x\to +\infty}\varphi_{q,F}(x)=0$.
We now introduce
\begin{align} \label{eq:psi-sigma-affine}
 \psi_q(x)(x)&=\kappa\PAR{1+ \frac{\sigma_1}{\sigma_0}x}^{-a_0+b_0-1}\times\frac{_1F_1\PAR{1+a_0-b_0,2-b_0;\frac{2\PAR{c_0\sigma_1-c_1\sigma_0}}{\sigma_1^2\PAR{\sigma_0+\sigma_1x} }}}{_1F_1\PAR{1+a_0-b_0,2-b_0;\frac{2\PAR{c_0\sigma_1-c_1\sigma_1}}{\sigma_1^2\sigma_0 }}}-\kappa\varphi_0^q(x),
\end{align}
where  $\kappa$ is such that $ \PAR{\psi_q}'(0)=1$, and $a_0:=a_F$, $b_0:=b_F$, and $\varphi^q_0:=\varphi_{q,F}$ for $F\equiv 0$. Note that $\kappa$ has an explicit (but not nice) expression in terms of the model parameters. 
We can also easily compute that
\[
\varphi_{q,F}'(0)=
-a_F\frac{\sigma_1}{\sigma_0}-\frac{2\PAR{c_1\sigma_0-c_0\sigma_1}a_F}{\sigma_1\sigma_0^2b_F}\times
\frac{_1F_1\PAR{1+a_F,1+b_F;\frac{2\PAR{c_0\sigma_1-c_1\sigma_0}}{\sigma_1^2\sigma_0 }}}
{_1F_1\PAR{a_F,b_F;\frac{2\PAR{c_0\sigma_1-c_1\sigma_0}}{\sigma_1^2\sigma_0 }}}.
\]

\subsubsection{When $\sigma^2$ is an affine function}

 Assume $\sigma(x)=\sqrt{\sigma_0+\sigma_1x}$, with $\sigma_0, \sigma_1> 0$. For example, if $\sigma_0=0$, then $X$ and $Y$ are square-root or Cox–Ingersoll–Ross processes. We have
\[
\PAR{\cL_F -q}u=\frac{1}{2}\PAR{\sigma_0+\sigma_1 x}u''+\PAR{c_0+c_1x}u'-qu.
\]

Let $a_F=\frac{q}{|c_1|}$ and $b_F=1+2\frac{c_1\sigma_0-c_0\sigma_1}{\sigma_1^2}$. Then, when $c_1>0$,
\begin{align}\label{eq:phi-sigma2-affine+}
\varphi_{q,F}(x)=\mathrm{e}^{-2\frac{c_1}{\sigma_1}x}\PAR{1+\frac{\sigma_1}{\sigma_0} x}^{b_F}\times\frac{U\PAR{1+a_F,1+b_F;\frac{2c_1(\sigma_0+\sigma_1x)}{\sigma_1^2}}}{U\PAR{1+a_F,1+b_F;\frac{2c_1\sigma_0}{\sigma_1^2}}},
\end{align}
and, when $c_1<0$,
\begin{align*}
\varphi_{q,F}(x)
=\PAR{1+\frac{\sigma_1}{\sigma_0}x}^{b_F}\times\frac{U\PAR{a_F+b_F,1+b_F;-\frac{2c_1(\sigma_0+\sigma_1x)}{\sigma_1^2}}}{U\PAR{a_F+b_F,1+b_F;-\frac{2c_1\sigma_0}{\sigma_1^2}}}
=\frac{U\PAR{a_F,1-b_F;-\frac{2c_1(\sigma_0+\sigma_1x)}{\sigma_1^2}}}{U\PAR{a_F,1-b_F;-\frac{2c_1\sigma_0}{\sigma_1^2}}}.
\end{align*}
Using~\eqref{eq:U-integral}, we easily observe that $\varphi_{q,F}(0)=1$ and $\lim_{x\to +\infty}\varphi_{q,F}(x)=0$.
Moreover, we have, when $\mu_1>0$,
\begin{align}\label{eq:psi-sigma2-affine+}
   \psi_q(x)&= \kappa \mathrm{e}^{-2\frac{\mu_1}{\sigma_1}x}\PAR{1+\frac{\sigma_1}{\sigma_0} x}^{b_0}\times \frac{_1F_1\PAR{1+a_0,1+b_0,\frac{2\mu_1(\sigma_0+\sigma_1x)}{\sigma_1^2}}}{_1F_1\PAR{1+a_0,1+b_0,\frac{2\mu_1\sigma_0}{\sigma_1^2}}}-\kappa \varphi^q_0(x),
\end{align}
and, when $\mu_1<0$,
\begin{align*}
   \psi_q(x)
   &=   \kappa \PAR{1+\frac{\sigma_1}{\sigma_0} x}^{b_0}\times \frac{_1F_1\PAR{a_0+b_0, b_0,-\frac{2\mu_1(\sigma_0+\sigma_1x)}{\sigma_1^2}}}{_1F_1\PAR{a_0+b_0,b_0,-\frac{2\mu_1\sigma_0}{\sigma_1^2}}}-\kappa \varphi^q_0(x),
\end{align*}
where $\kappa$ is explicitly determined by $\PAR{\psi_q}'(0)=1$, and $a_0:=a_F$, $b_0:=b_F$, and $\varphi^q_0:=\varphi_{q,F}$ for $F\equiv 0$. 
When $c_1=0$ or $\mu_1=0$, the functions $\varphi_{q,F}$ and $\psi_q$ have also an explicit expression using modified Bessel functions.

We easily compute, when $c_1>0$,
\[
\varphi_{q,F}'(0)=\frac{\sigma_1-2c_0}{\sigma_0}-2\frac{c_1+q}{\sigma_1}\,
\frac{U\PAR{2+a_F,2+b_F;\frac{2c_1\sigma_0}{\sigma_1^2}}}{U\PAR{1+a_F,1+b_F;\frac{2c_1\sigma_0}{\sigma_1^2}}},
\]
and, when $c_1<0$,
\[
\varphi_{q,F}'(0)=b_F\frac{\sigma_1}{\sigma_0}-2\frac{q}{\sigma_1}\frac{U\PAR{1+a_F,2-b_F;-\frac{2c_1\sigma_0}{\sigma_1^2}}}{U\PAR{a_F,1-b_F;-\frac{2c_1\sigma_0}{\sigma_1^2}}}.
\]

\bigskip

In Figure~\ref{graph:phi-psi}, the curves of $\varphi_{q,0}$ and $\psi_q$ are drawn in the above three cases for a given set of parameters. We observe the properties announced in Lemma~\ref{lemma:concave-convex}, i.e., when $\mu_1<q$, $\varphi_{q,0}$ is strictly convex and $\psi_q$ is concave-convex.
\begin{figure}[h!]
\centering
\begin{tabular}{cccc}
      \includegraphics[scale=0.3]{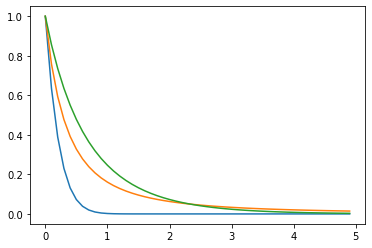}&
      \includegraphics[scale=0.3]{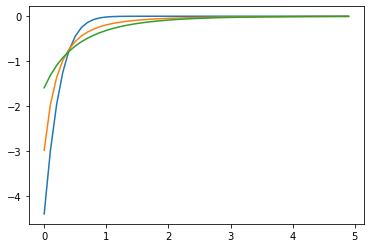}
     &\includegraphics[scale=0.3]{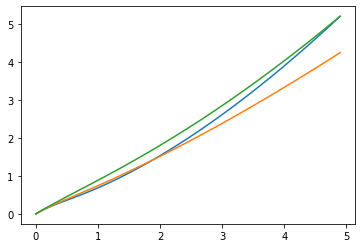}&
      \includegraphics[scale=0.3]{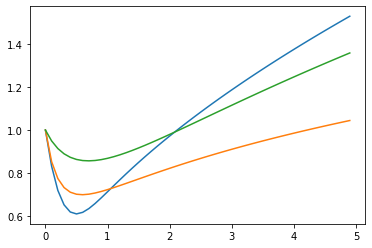}
     \\ [-0.2 cm]
     \small{$\varphi_{q,0}$}&\small{$\varphi'_{q,0}$} &
     \small{$\psi_q$}& \small{$\psi'_q$}
\end{tabular}
\caption{The functions $\varphi_{q,0}$ and $\psi_q$ and their first derivatives when $\mu_0=0.09,\, \mu_1=0.21, \, \sigma_0=0.3,\, \sigma_1=0.5$ and $q=0.33$: in blue when $\sigma(x)=\sigma_0$, in orange when $\sigma(x)=\sigma_0+\sigma_1 x$, and in green when $\sigma^2(x)=\sigma_0+\sigma_1 x$.}
\label{graph:phi-psi}
\end{figure}

We now compute $b^\ast$ numerically, as the root of Equation~\eqref{eq:def-bstar}, and the value function $J_{b^\ast}$, given by Equation~\eqref{J_optimal}, for different values of $(F_0,F_1)$. 
In Figure~\ref{graph:heatmap}, we observe the dependence of $b^*$ on $(F_0,F_1)\in\mathbb{R}_+^2$, the other parameters being fixed. The  white area below the red curve represents the case $b^*=0$, and the area above the red curve coincides with the case $b^*>0$: the darker the blue, the higher the value of $b^*$. 
\begin{figure}[h!]
\centering
\begin{tabular}{ccc}
      \includegraphics[scale=0.4]{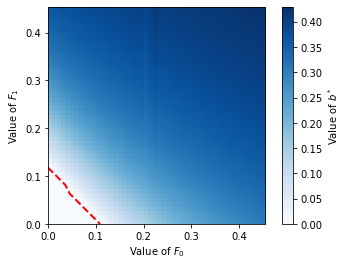}&
      \includegraphics[scale=0.4]{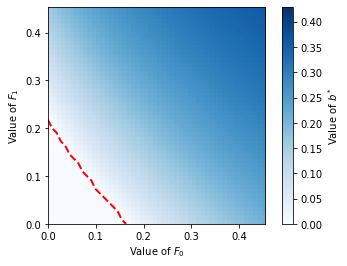} &
      \includegraphics[scale=0.4]{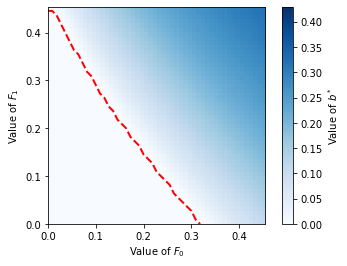}
     \\ 
     \small{when $\sigma(x)=\sigma_0$}& \small{when $\sigma(x)=\sigma_0+\sigma_1 x$}& \small{when $\sigma^2(x)=\sigma_0+\sigma_1 x$}
\end{tabular}
\caption{Heat maps of $b^*$ as a function of $(F_0,F_1)\in\mathbb{R}_+^2$, when $\mu_0=0.09,\, \mu_1=0.21, \, \sigma_0=0.3,\, \sigma_1=0.5$ and $q=0.33$.}
\label{graph:heatmap}
\end{figure}

Using~\eqref{J_optimal}, and the curves of $J_{b^\ast}$ and their corresponding first derivatives are drawn in Figure~\ref{graph:value} for different values of $(\sigma_0,\sigma_1)$. In all cases, we observe that the first derivative is decreasing which coincides with the fact that $J_{b^\ast}$ is concave on $\R_+$.

\begin{figure}[h!]
\centering
\begin{tabular}{ccc}
      \includegraphics[scale=0.35]{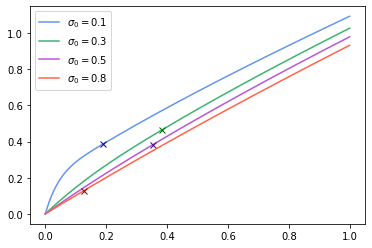}&
      \includegraphics[scale=0.35]{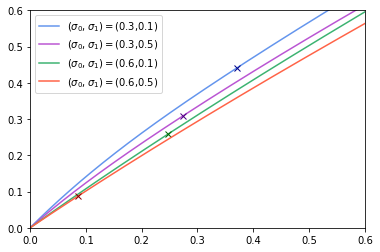} &
      \includegraphics[scale=0.35]{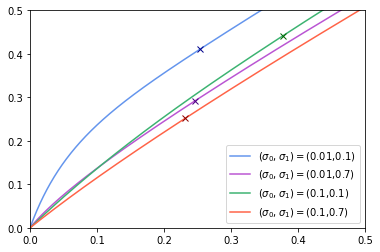}
      \\[-0.2cm]
     \small{$\sigma(x)=\sigma_0$}& \small{$\sigma(x)=\sigma_0+\sigma_1 x$}& \small{$\sigma^2(x)=\sigma_0+\sigma_1 x$}
      \\
       \includegraphics[scale=0.35]{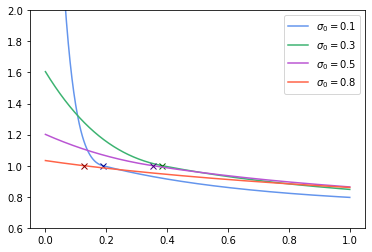}&
      \includegraphics[scale=0.35]{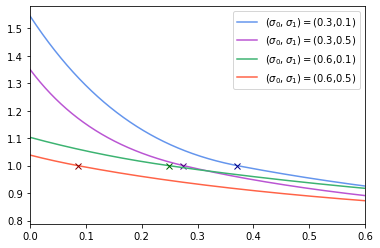} &
      \includegraphics[scale=0.35]{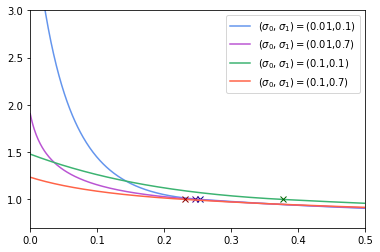}
\end{tabular}

\caption{The value function $J_{b^\ast}$ (on the first row) and its first derivative (on the second row) when $\mu_0=0.09,\, \mu_1=0.21, \, F_0=0.3,\, F_1=0.3$ and $q=0.33$, and for different couples $(\sigma_0,\sigma_1)$. The optimal refraction threshold is marked with a cross.}
\label{graph:value}
\end{figure}
\subsection{A logistic-type model}

We consider a modification of the logistic model, where the drift coefficient is taken as $\mu(x)=\mu_0+\mu_1x(1-x/K)$ and the diffusion coefficient is given by $\sigma(x)=\sqrt{\sigma_0+\sigma_1 x}$. In the literature of logistic models, the parameter $\mu_1$ is known as the \textit{population growth rate}, while the parameter $K$ is known as the \textit{carrying capacity of the environment} (see, e.g., \cite[Section 2]{alvarez-shepp_1998}). Note that the parameters $\mu_0 >0$ and $\sigma_0>0$ are important in order to satisfy Assumptions \ref{assumptions}. We also remark that, given a discount factor $q>0$, the value of $\mu_1$ is chosen so that $\mu_1 \leq q$, as one can easily verify that $\mu'(0) = \mu_1$. Throughout this section, we will consider that the bound function is of the form $F(x) = F_0 + F_1x$ for all $x \in \R$.

While in this case we do not have semi-explicit expressions for the functions $\psiq, \phiqF, \text { and } I_F$, we are still able to perform computations by solving numerically their associated ODEs, this allows us to find numerically $\hat{b}$ the unique inflection point of $\psiq'$. Consequently, we can find a numerical solution to Equation~\eqref{b_optimal}.

In Figure \ref{logistic:heatmap}, we illustrate the sensitivity of $b^\ast$ for different values of $(F_0,F_1)$. In particular, the region below the red curve represents the values of $(F_0,F_1)$ for which $b^\ast = 0$; whereas the region above the red curve shows the values where $b^\ast > 0$. We remark that the darker the shade, the larger the value of $b^\ast$. We can observe that, for a fixed value of $F_0$, $b^\ast$ is nondecreasing with respect to $F_1$; and, likewise for a fixed value of $F_1$, then $b^\ast$ is nondecreasing with respect to $F_0$. Also notice that, thanks to Proposition \ref{prop_exist_b}, we have that $b^\ast$ has an upper bound given by $\hat{b}$, the unique inflection point of $\psiq'$, and in Figure \ref{logistic:heatmap} we can observe that the region with the darkest shade corresponds to the values of $(F_0,F_1)$ where $b^\ast$ is close to the value $\hat{b}$.\\
\begin{figure}[h!]
    \centering
    \includegraphics[scale=0.5]{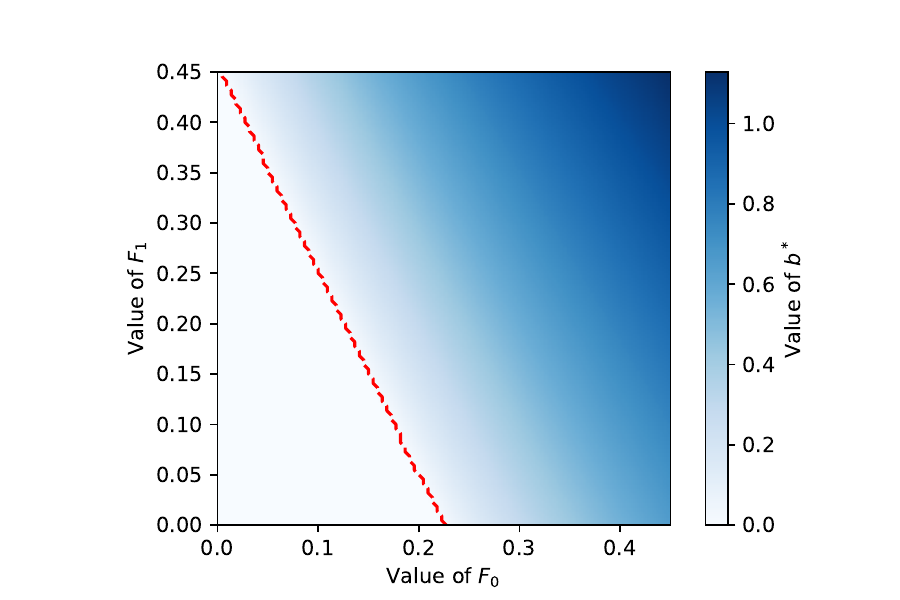}
    \caption{Heat map of $b^\ast$ as a function of $(F_0,F_1) \in \R_+^2$, when $\mu_0 = 0.15, \, \mu_1 = 0.21, \, K = 10, \, \sigma_0 = 0.75,\, \sigma_1 = 0.5, \text{ and } q = 0.33$.}
    \label{logistic:heatmap}
\end{figure}

In Figure \ref{logistic:sigma:effect}, we illustrate the effect of the diffusion coefficient on the value function, as well as in the optimal refraction barrier. To this end, first we fixed a value of $\sigma_1$, and proceeded to solve the optimisation problem for different values of $\sigma_0$; afterwards, we fixed a value of $\sigma_0$ and solved the control problem for different values of $\sigma_1$. Firstly, we confirm the concavity of the value function on all cases. Regarding the effect of $\sigma_0$ and $\sigma_1$, we notice that if $\sigma_0 \leq \widetilde{\sigma}_0$ while $\sigma_1$ is fixed, then the corresponding value functions $J_{b^\ast}$ and $J_{\widetilde{b}^\ast}$ satisfy $J_{\widetilde{b}^\ast} \leq J_{b^\ast}$. We observe a similar effect with respect to $\sigma_1$ while $\sigma_0$ is fixed.

\begin{figure}[h!]
    \centering
    \begin{tabular}{cc}
        \includegraphics[scale=0.5]{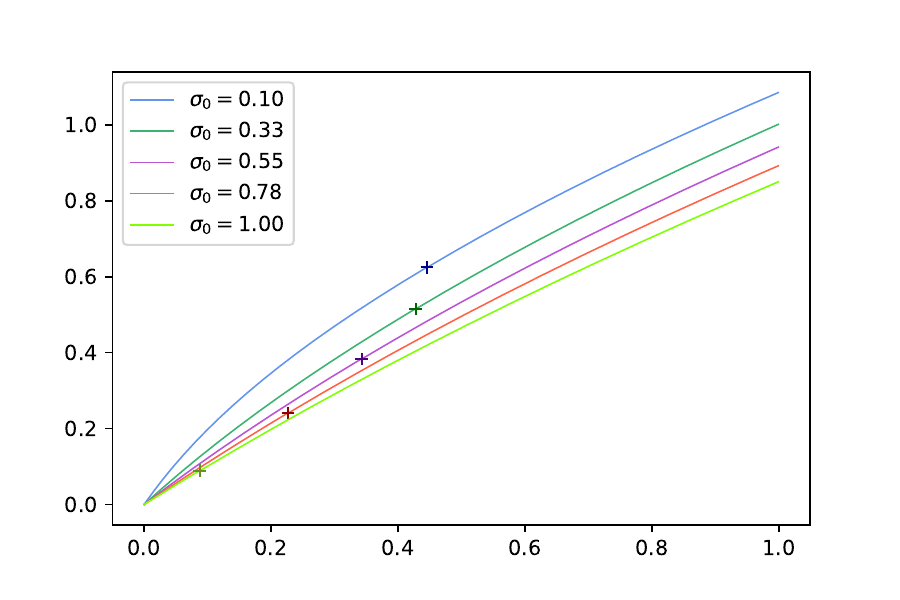}     &    \includegraphics[scale=0.5]{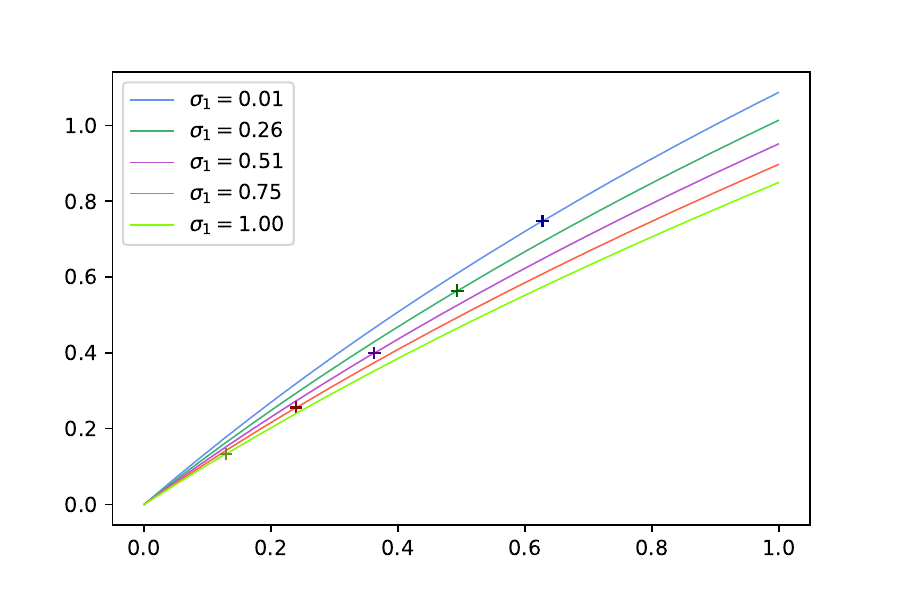} \\
         \includegraphics[scale=0.5]{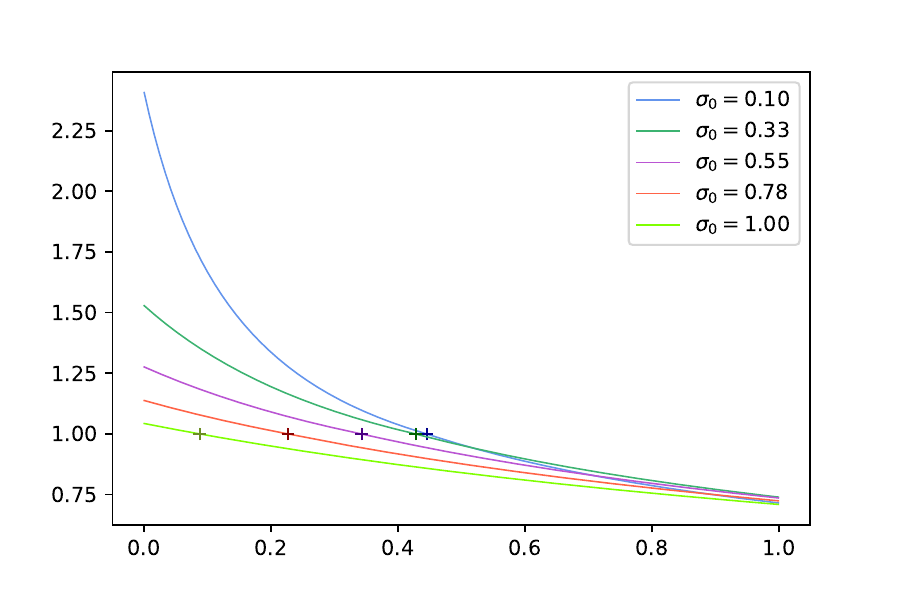} &   \includegraphics[scale=0.5]{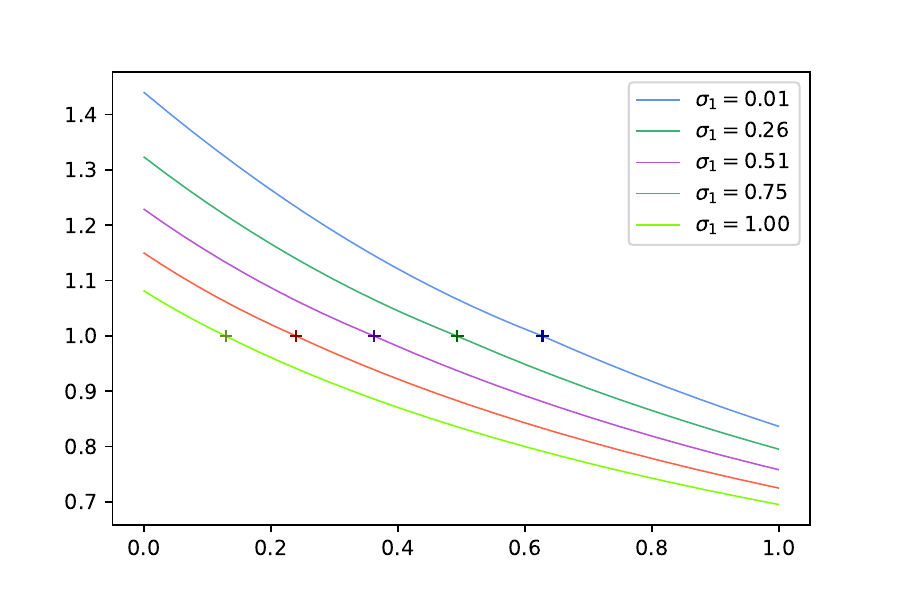}
    \end{tabular}

    \caption{Effect of the values of $\sigma_0$ and $\sigma_1$ when $q = 0.33, \,\mu_0= 0.25, \, \mu_1 = 0.3, \, K = 5 $. The bound function is $F(x) = 0.15 + 0.25 x$. The optimal refraction threshold is marked as a cross. Left: Value function $J_{b^\ast}$ (on the top row) and its first derivative (on the second row) for different values of $\sigma_0$ when $\sigma_1 = 0.75$ is fixed. Right: Value function $J_{b^\ast}$ (on the top row) and its first derivative (on the second row) for different values of $\sigma_1$ when $\sigma_0 = 0.75$ is fixed. }
    \label{logistic:sigma:effect}
\end{figure}

\section*{Acknowledgements}

The authors would like to thank Clarence Simard for fruitful discussions.

Funding in support of this work was provided by a CRM-ISM Postdoctoral Fellowhip and Discovery Grants (RGPIN-2019-06538, RGPIN-2020-07239) from the Natural Sciences and Engineering Research Council of Canada (NSERC).

\appendix

\section{Finiteness and transversality property}\label{App:Defined}

Here is the proof of Proposition~\ref{Prop:Defined}.

\begin{proof}
    Let $h(x) = 2(1 - |x|\wedge 1)$ for $x \in \R$, and note that $h$ is continuous on $\R$, $h(x) \in [0,2]$ for all $x\in \R$ and $h(x) = 0$ whenever $|x|> 1$. Next, we construct the functions
    \[
    g(x) = \begin{cases}
        -1, & x < -1 \\
        -1 + \int_{-1}^x h(y)\diff y, & x \in [-1,1] \\
        1, & x > 1
    \end{cases}, \quad f(x) = \begin{cases}
        -x, & x < -1 \\
        1 + \int_{-1}^x g(y)\diff y, & x \in [-1,1] \\
        x, & x > 1
    \end{cases}.
    \]
    Note that, by their construction, we have $f(x) \geq |x|, \, f'(x) = g(x), \, f'(x) = \text{sgn}(x) $ whenever $|x| > 1$; and, $f''(x) = h(x)$. Here $\text{sgn}$ denotes the \textit{sign} function.\\
    Define the constants $\tilde{q}:= \max(\mu'(0+),q/2), \bar{\sigma} := \sup_{|x| < 1}\sigma(x)$, and $$C =  \max\left\lbrace \sup_{|x|<1} \left[- \tilde{q}f(x) + f'(x)(\mu(x)-F(x))\right] ,0 \right\rbrace.$$ Note that $\tilde q > 0 $ and $\mu'(0+) \leq \tilde{q} < q.$ Because $f$ is smooth enough and $f''$ is bounded, we can apply It\^{o}'s Formula in order to obtain

    $\E_x \left[ e^{-\tilde{q}t} f(X_t) \right] - f(x)$
    \begin{align}
        & =  \E_x\left[ \int_0^t e^{-\tilde{q}s} \left( -\tilde{q} f(X_s) + f'(X_s)(\mu(X_s) - F(X_s)) + \frac{1}{2}\sigma^2(X_s) f''(X_s) \right) \diff s  \right] \notag\\
        &\leq  \E_x\left[ \int_0^t e^{-\tilde{q}s} \left( \Ind_{\lbrace 
|X_s| > 1 \rbrace}(-\tilde{q} f(X_s) + f'(X_s)(\mu(X_s) - F(X_s))) + C \Ind_{\lbrace 
|X_s| \leq 1 \rbrace} + \bar{\sigma}^2\right) \diff s  \right] \notag\\
& =   \E_x\left[ \int_0^t e^{-\tilde{q}s} \left( \Ind_{\lbrace 
|X_s| > 1 \rbrace}(-\tilde{q} |X_s| + \text{sgn}(X_s)(\mu(X_s) - F(X_s))) + C \Ind_{\lbrace 
|X_s| \leq 1 \rbrace} + \bar{\sigma}^2\right) \diff s  \right]. \label{defined_0}
    \end{align}
Observe that the term $\E_x\left[ \int_0^t e^{-\tilde{q}s}\left( C \Ind_{\lbrace 
|X_s| \leq 1 \rbrace} + \bar{\sigma}^2 \right) \diff s \right]$ is non-negative, and we have the upper bound 
\begin{equation}\label{defined_1}
\E_x\left[ \int_0^t e^{-\tilde{q}s}\left( C \Ind_{\lbrace 
|X_s| \leq 1 \rbrace} + \bar{\sigma}^2 \right) \diff s \right] \leq \frac{C + \bar{\sigma}^2}{\tilde{q}} .
\end{equation}
It remains to analyse the term $ \E_x\left[ \int_0^t e^{-\tilde{q}s} \Ind_{\lbrace 
|X_s| > 1 \rbrace}(-\tilde{q} |X_s| + \text{sgn}(X_s)(\mu(X_s) - F(X_s))) \diff s\right]$. To this end, if $z > 1$ we have that 
\begin{equation}\label{defined_2}-\tilde{q} z + (\mu(z) - F(z)) \leq -\tilde{q} z + \mu(z) \leq \mu(0),
\end{equation}
where in the first inequality we used that $F$ is non-negative on $(0,\infty)$, and in the second inequality we used that $\mu$ is concave, hence $\mu(z) \leq \mu(0) + \mu'(0+) z$, together with the fact that $\tilde{q} \geq \mu'(0+)$.\\
On the other hand, if $z < -1$ we have
\begin{equation}\label{defined_3} \tilde{q} z - (\mu(z) - F(z)) = ( \tilde{q} - \mu'(0+) + F'(0+) )z + F(0) - \mu(0) \leq F(0) - \mu(0) \leq F(0), \end{equation}
where in the equality we simply substituted the value of the extensions of $\mu$ and $F$, in the first inequality we used that $\tilde{q} \geq \mu'(0+)$ and $F'(0+)z \leq 0$ because $F'(0+) \geq 0$, and in the second inequality we used that $\mu(0) >0$.\\
Using the upper bounds \eqref{defined_2} and \eqref{defined_3}, we have that 
\begin{align}
\E_x\left[ \int_0^t e^{-\tilde{q}s} \Ind_{\lbrace 
|X_s| > 1 \rbrace}(-\tilde{q} |X_s| + \text{sgn}(X_s)(\mu(X_s) - F(X_s))) \diff s\right] &\leq \E_x\left[ \int_0^t e^{-\tilde{q}s} \Ind_{\lbrace 
|X_s| > 1 \rbrace} \mu(0) \vee F(0) \diff s\right] \notag \\
&\leq \frac{ \mu(0) \vee F(0) }{\tilde{q}} .
\label{defined_4}
\end{align}

Hence, by combining \eqref{defined_1} and \eqref{defined_4} we have that
\[
\E_x \left[ e^{-\tilde{q}t} f(X_t) \right] \leq f(x) + \frac{ \mu(0) \vee F(0) + C + \bar{\sigma}^2 }{\tilde{q}} .
\]
Let 
\begin{equation}\label{H_x}
H_x:= f(x) + \frac{ \mu(0) \vee F(0) + C + \bar{\sigma}^2 }{\tilde{q}} .
\end{equation}
Then, since $f(x) \geq |x|$ we get the upper bound
\begin{equation}\label{defined_5}
 \E_x \left[ e^{-\tilde{q}t} |X_t| \right] \leq \E_x \left[ e^{-\tilde{q}t} f(X_t) \right] \leq H_x.
\end{equation}
Finally, using that $F$ is concave, and \eqref{defined_5}, we obtain for $q > \mu'(0+)$ 
\begin{align*}
    \E_x\left[ \int_0^{\infty} e^{-qt} |F(X_t)| \diff t \right] &\leq \E_x\left[ \int_0^{\infty} e^{-qt} (F(0) + F'(0)|X_t|) \diff t \right] \\
    & \leq \frac{F(0)}{q} + F'(0) \int_0^{\infty} e^{-(q-\tilde{q})t} \E_x\left[ e^{-\tilde{q}t} |X_t| \right]\diff t \\
    & \leq \frac{F(0)}{q} + F'(0) \frac{H_x}{q - \tilde{q}} < \infty.
\end{align*}
This proves the result.
\end{proof}

From the last proof, we can extract the following lemma.

\begin{lemma}\label{rem:transversal}
Under Assumptions~\ref{assumptions}, we have the following transversality property: $\E_x\left[ e^{-qT} X_T \right] \to 0$ as $T \to \infty$.
\end{lemma}
\begin{proof}
Denoting $\widetilde q=\mu'(0+)$, it is proved in the proof of Proposition~\ref{Prop:Defined} that 
\[
\E_x\left[ e^{-qT} |X_T| \right] = e^{-(q-\tilde{q})T} \E_x\left[ e^{-\tilde{q}T} |X_T| \right] \leq e^{-(q-\tilde{q})T} H_x ,
\]
where $H_x$ is defined in \eqref{H_x}. Since $q > \mu'(0+)$ by Assumptions~\ref{assumptions}, we have $e^{-(q-\tilde{q})T} H_x \rightarrow 0$, as $T \rightarrow\infty$, for all $x \in \R$.
\end{proof}

\section{Properties of the concave envelope}\label{App:Envelope}

Let $I$ be  a function defined on $\R$. Its concave envelope $U$ is defined, in broad terms, as the smallest concave function lying above $I$. A formal definition is given by
\begin{equation}\label{env3}
U(z) := \inf\BRA{ \Phi(z) : \Phi \text{ is concave and } \Phi \geq I}.
\end{equation}
The function $U$ takes its values in $\R\cup\BRA{+\infty}$.

There are various alternative and equivalent definitions of the concave envelope, we list some of them: for $z \in \R$,
\begin{align}
\label{env2}
    U(z)& = \sup \left\lbrace I(x) + \frac{I(x) - I(y)}{x-y}(z-x) : x \leq z \leq y \right\rbrace ;
\\
    U(z) &= \inf \BRA{ \Phi(z) : \Phi \text{ is affine and } \Phi \geq I } .
    \end{align}
In any case, $U$ is itself a concave function.  

We refer the reader to \cite[Chap. B, Sec. 2.5]{hiriart2004fundamentals} and \cite{roberts-varberg_1973} for a broader and more detailed exposition of \textit{convex envelopes}, and their properties.

\begin{lemma}\label{lem:affine}
Let $I$ be a continuous function on $\R$ and $U$ its concave envelope. Assume that $U$ takes finite values. If there exists $( a,b)\subset \R$ such that $I<U$ on $( a,b)$, then $U$ is affine on $(a,b)$.
\end{lemma}

\begin{proof}
We mainly follow the proof of \cite[Lemma 2.2]{MASTYLO2006192}. Let us define the following affine function:
\[
L(z)=U(a)+\frac{U(b)-U(a)}{b -a}(z-a).
\]
In particular, we have $L(a)=U(a)$ and $L(b)=U(b)$. Since $U$ is concave, we have $U\geq L$ on $[a,b]$ and $U\leq L$ on $\R\setminus [a,b]$. We also know that $U\geq I$ on $\R$ and we assume that $U>I$ on $(a,b)$. Therefore, as $U \geq L$ on $(a,b)$, it suffices to prove that $L \geq I$ on $(a,b)$.

Let us consider $I-L$ on $\PAR{a,b}$ and define
\[
m=\max_{z\in\PAR{a,b}}\BRA{I(z)-L(z)}.
\]
By continuity of the functions involved, there exists $z^\ast \in \SBRA{a,b}$, such that $m=I(z^\ast)-L(z^\ast)$. If $m<0$, the result follows.

If $m\geq 0$, as $I\leq U\leq L$ on $\R\setminus [a,b]$, which means $\max_{z\notin\PAR{a,b}}\BRA{I(z)-L(z)} \leq 0$, we can write
\[
m=\max_{z\in\R}\BRA{I(z)-L(z)}.
\]
Thus, for all $z\in\R$, $I(z)\leq m+L(z)$. As $U$ is the concave envelope of $I$ and $m+L$ is affine, by definition~\eqref{env3}, we deduce
\begin{equation}\label{eq:upper bound u}
 U(z)\leq m+L(z) , \quad \text{for all $z \in [a,b]$.}
\end{equation}
Then, for $z=z^\ast$,
\[
I(z^\ast)\leq U(z^\ast)\leq m+L(z^\ast)=I(z^\ast)
\]
and we deduce that $I(z^\ast)=U(z^\ast)$.

Since by assumption we have $U>I$ on $\PAR{a,b}$, we must have that $z^\ast=a$ or $z^\ast=b$. So, either $I(a)-L(a)= U(a)-L(a)=0$ or $I(b)-L(b)= U(b)-L(b)=0$. Thus, $m=0$ and the result follows.
\end{proof}

\begin{lemma}\label{lemma:interval}
Let $I$ be a function in $\mathcal{C}^2(\R)$ and $U$ its concave envelope. Let $x_0\in\R$ such that $I(x_0)<U(x_0)<\infty$.

\begin{enumerate}
\item \label{casex} If there exists $\hat x \in (-\infty, x_0)$ such that $U(\hat x)=I(\hat x)$ and $I<U$ on $(\hat x, x_0)$, then $U$ is differentiable at $\hat x$ and
\[U(\hat x)=I(\hat x), \quad U'(\hat x)=I'(\hat x),\quad \text{ and}\quad  I''(\hat x)\leq 0,\]

\item \label{casey}
If there exists $\hat y \in (x_0,+\infty)$ such that $U(\hat y)=I(\hat y)$ and $I<U$ on $(x_0, \hat y)$, then $U$ is differentiable at $\hat y$ and
\[U(\hat y)=I(\hat y), \quad U'(\hat y)=I'(\hat y),\quad \text{ and}\quad I''(\hat y)\leq 0.\]
\end{enumerate}

\end{lemma}

\begin{proof}
Assume case (\ref{casex}) is satisfied.

From Lemma~\ref{lem:affine}, we know that $U$ is affine on $[\hat x, x_0]$. We denote by $L$ the affine function
\begin{align*}
L(x)&=U(\hat x)+U'(\hat x+)(x-\hat x)
\\
&
=I(\hat x)+U'(\hat x+)(x-\hat x),
\end{align*}
where $U'(\hat x+)=\lim_{\underset{x>\hat x}{x\to \hat x}}U'(x)$ is the right derivative of $U$ at $\hat x$.

We have $L=U$ on $[\hat x,x_0]$, and $U\leq L$ on $\R$ by concavity of $U$. We prove by contradiction that $U'(\hat x+)=I'(\hat x)$. 
If $I'(\hat x)>U'(\hat x+)$, then there is $x>\hat x$ such that $I(x)>U(x)$, which is impossible since $U\geq I$ by definition of the concave envelope.
If $I'(\hat x)<U'(\hat x+)$, then there is $x<\hat x$ such that $I(x)>L(x)$, which is also impossible since $U\leq L$.
Then $U'(\hat x+)=I'(\hat x)$.

Since $I\in \mathcal{C}^2(\R)$, by Taylor's theorem, there exists $\varepsilon$ such that, for $x\geq \hat x$,
\begin{align*}
I(x)&=I(\hat x)+I'(\hat x)(x-\hat x)+\frac{I''(\hat x)}{2}(x-\hat x)^2+\varepsilon(x)(x-\hat x)^2\\
&
=U(x)+\frac{I''(\hat x)}{2}(x-\hat x)^2+\varepsilon(x)(x-\hat x)^2,
\end{align*}
with $\lim_{x\to \hat x}\varepsilon(x)=0$. As $U\geq I$, we deduce that $I''(\hat x)\leq 0$.

Finally, let us note that on a small interval of the form $[\tilde x,\hat x]$, with $\tilde x<\hat x$, we have either $U(x)=I(x)$, and then $U$ is differentiable in $\hat x$, with  $U'(\hat x)=I'(\hat x)$, or  $U(x)>I(x)$, and then by the previous proof $U'(\hat x-)=I'(\hat x)$.  Case (\ref{casey}) of the lemma is proved in a very similar way.
\end{proof}

 \section{Fundamental solutions to the ODE for specific models}\label{App:proof-example}

We prove in this section the expressions of the fundamental solutions given in Section~\ref{sec:affine-examples}, when $\mu$ and $F$ are affine functions.
We use the same notations  as in Section~\ref{sec:examples}.

\medskip
\paragraph{\bf When $\sigma$ is constant}
For $c_1>0$, the expression of $\varphi_{q,F}$ can be found in \cite[Equation 2.0.1, Chapter 7]{borodin2002handbook}, and for $c_1=0$, the result is obvious. We give here a proof for $c_1\neq 0$.

By \cite[Chapter 19]{abramowitz-stegun_1964}, two linearly independent solutions of 
\begin{equation}\label{eq:ODE-parabolic}
y''-\PAR{a+\frac{z^2}{4}}y=0
\end{equation}
are
\[
y_1(z)=\mathrm{e}^{-\frac{z^2}{4}}\,_1F_1\PAR{\frac{1}{2}a+\frac{1}{4},\frac{1}{2};\frac{z^2}{2}}
\quad 
\text{and}
\quad
y_2(z)=z\mathrm{e}^{-\frac{z^2}{4}}\,_1F_1\PAR{\frac{1}{2}a+\frac{3}{4},\frac{3}{2};\frac{z^2}{2}}.
\]
The Parabolic cylinder function $D_{-\lambda}$ introduced in Section~\ref{sec:affine-examples} is indeed a particular linear combination of $y_1$ and $y_2$ with $\lambda=a+\frac{1}{2}$ (see \cite[Equation 19.3.1]{abramowitz-stegun_1964}).

We will prove that 
\[
u(x)=\mathrm{e}^{-\frac{z^2}{2c_1}}y(\alpha z),
\]
with $z=\frac{c_0+c_1 x}{\sigma}$, $\alpha=\text{sgn}(c_1)\sqrt{\frac{2}{|c_1|}}$, and $y$ solution of \eqref{eq:ODE-parabolic} with $a=\frac{q}{\ABS{c_1}}+\frac{\text{sgn}(c_1)}{2}$, is a solution of
\begin{equation}\label{eq:ODE-sigma-constant}
\frac{\sigma^2}{2}u''+\PAR{c_0+c_1x}u'-qu=0.
\end{equation}
We can then deduce that  the fundamental solutions of \eqref{eq:ODE-sigma-constant} are linear combinations of $u_1(x)=\mathrm{e}^{-\frac{z^2}{2c_1}}y_1(\alpha z)$ and $u_2(x)=\mathrm{e}^{-\frac{z^2}{2c_1}}y_2(\alpha z)$. 
As $\varphi_{q,F}$ must satisfies $\varphi_{q,F}(0)=1$ and $\lim_{x\to+\infty}\varphi_{q,F}(x)=0$, from Expression~\eqref{eq:D-integral} of the parabolic cylinder function, we deduce the expression of $\varphi_{q,F}$ given by Equation~\eqref{eq:phi-sigma-constant}. The second fundamental solution $\psi_{q}$ is thus a linear combination of $\varphi_{q,0}$ and $u_1$ (or $\varphi_{q,0}$ and $u_2$), with coefficients determined by $\psi_{q}(0)=0$ and $\psi'_{q}(0)=1$. We easily observe that the expression~\eqref{eq:psi-sigma-constant} given in Section~\ref{sec:affine-examples} satisfies the good boundary conditions.

Let us now prove that $u(x)=\mathrm{e}^{-\frac{z^2}{2c_1}}y(\alpha z)$ with $\alpha$ and $z$ defined above is a solution of \eqref{eq:ODE-sigma-constant}. We compute the first and second derivative of $u$ with respect to $x$:
\begin{align*}
u'(x)&=-\frac{c_0+c_1x}{\sigma^2}u(x)+\frac{\sqrt{2\ABS{c_1}}}{\sigma}\mathrm{e}^{-\frac{z^2}{2c_1}}y'(\alpha z)\\
u''(x)&=-\frac{c_1}{\sigma^2}u(x)+\frac{(c_0+c_1x)^2}{\sigma^4}u(x)-2\frac{\sqrt{2\ABS{c_1}}}{\sigma^3}(c_0+c_1x)\mathrm{e}^{-\frac{z^2}{2c_1}}y'(\alpha z)+\frac{2\ABS{c_1}}{\sigma^2}\mathrm{e}^{-\frac{z^2}{2c_1}}y''(\alpha z).
\end{align*}
Consequently,
\begin{align*}
\frac{\sigma^2}{2}u''(x)+\PAR{c_0+c_1x}u'(x)-qu(x)
&=-\PAR{q+\frac{c_1}{2}+\frac{(c_0+c_1x)^2}{2\sigma^2}}u(x)+\ABS{c_1}\mathrm{e}^{-\frac{z^2}
{2c_1}}y''(\alpha z)
\\
&=\ABS{c_1}\mathrm{e}^{-\frac{z^2}
{2c_1}}\PAR{y''(b z)-\PAR{a+\frac{\alpha^2z^2}{4}}y(\alpha z)}=0.
\end{align*}
The function $u(x)=\mathrm{e}^{-\frac{z^2}{2c_1}}y(\alpha z)$ is a solution of the ODE~\eqref{eq:ODE-sigma-constant}, and the conclusion follows.

\medskip 
\paragraph{\bf When $\sigma$ is affine}

The proof scheme is the same as in the previous case. By \cite[Chapter 13]{abramowitz-stegun_1964}, two linearly independent solutions of the differential equation {\bf E$(a,b)$} defined by
\begin{equation}\label{eq:ODE-hypergoem}
zy''+(b-z)y'-ay=0.
\end{equation}
are  the Kummer confluent Hypergeometric function $_1F_1(a,b;z)$ of the first kind, and the confluent hypergeometric function of the second kind $U(a,b;z)$, where $U(a,b;z)$ is indeed an explicit linear combination of $_1F_1(a,b;z)$ and $z^{1-b}\,_1F_1(1+a-b,2-b;z)$ (see \cite[Equation 13.1.3]{abramowitz-stegun_1964}). 

We will prove that 
\[
u(x)=z^ay(\alpha z),
\]
with $z=(\sigma_0+\sigma_1 x)^{-1}$, $a=(\Delta-\sigma_1^2+2c_1)/(2\sigma_1^2)$, $b=1+\Delta/\sigma_1^2$, $\Delta^2=\PAR{\sigma_1^2-2c_1}^2+8q\sigma_1^2$, $\alpha=\frac{2(c_0\sigma_1-c_1\sigma_0)}{\sigma_1^2}$, and $y$ a solution of {\bf E$(a,b)$}, is a solution of
\begin{equation}\label{eq:ODE-sigma-affine}
\frac{1}{2}\PAR{\sigma_0+\sigma_1 x}^2u''+\PAR{c_0+c_1x}u'-qu=0
\end{equation}
As $a,b>0$, from the integral representation~\eqref{eq:1F1-integral} of $_1F_1$, we deduce that $\varphi_{q,F}$ defined by Equation~\eqref{eq:phi-sigma-affine} satisfies $\varphi_{q,F}(0)=1$ and $\lim_{x\to+\infty}\varphi_{q,F}(x)=0$. The function $\psi_{q}$ is thus a linear combination of $\varphi_{q,0}$ and $z^aU(a,b;\alpha z)$ where the coefficients are determined by the boundary conditions of the second fundamental solutions. In fact, we used $z^{1-b}\,_1F_1(1+a-b,2-b;z)$ instead of $U$ as second solution of {\bf E$(a,b)$} in Expression~\eqref{eq:psi-sigma-affine} of Section~\ref{sec:affine-examples}.

Let us now prove that $u(x)=z^ay(\alpha z)$, with $z$, $a$, $b$, and $\alpha$ defined above, is a solution of the ODE~\eqref{eq:ODE-sigma-affine}. We compute the first and second derivative of $u$ with respect to $x$:
\begin{align*}
u'(x)&=-a\sigma_1z^{a+1}y(\alpha z)-\sigma_1 \alpha z^{a+2}y'(\alpha z)\\
u''(x)&=a(a+1)\sigma_1^2z^{a+2}y(\alpha z)+2(a+1)\sigma_1^2\alpha z^{a+3}y'(\alpha z)+\sigma_1^2\alpha^2z^{a+4}y''(\alpha z).
\end{align*}
Consequently,
\begin{align*}
    &\frac{1}{2}\PAR{\sigma_0+\sigma_1 x}^2u''(x)+\PAR{c_0+c_1x}u'(x)-qu(x)\\
    &=\frac{1}{2z^2}u''(x)+\PAR{c_0+c_1x}u'(x)-qu(x)\\
    &=\PAR{\frac{a(a+1)}{2}\sigma_1^2-a\sigma_1\frac{c_0+c_1x}{\sigma_0+\sigma_1 x}-q}u(x)\\
    &\hskip 1cm +\SBRA{-(c_0+c_1x)+(a+1)\sigma_1\PAR{\sigma_0+\sigma_1 x}}\sigma_1\alpha z^{a+2}y'(\alpha z)+\frac{\sigma_1^2\alpha^2}{2}z^{a+2}y''(\alpha z).
\end{align*}
By definition of the parameter $a$, it is a root of the polynomial
\[
\sigma_1^2X^2+(\sigma_1^2-2c_1)X-2q=0,\]
which implies that $\frac{a(a+1)}{2}\sigma_1^2-ac_1-q=0$. We thus remark that
\begin{align*}
\frac{a(a+1)}{2}\sigma_1^2-a\sigma_1\frac{c_0+c_1x}{\sigma_0+\sigma_1 x}-q&=ac_1-a\sigma_1\frac{c_0+c_1x}{\sigma_0+\sigma_1 x}\\
&=-a\alpha\frac{\sigma_1^2}{2}z.
\end{align*}
Besides,
\begin{align*}
-(c_0+c_1x)+(a+1)\sigma_1\PAR{\sigma_0+\sigma_1 x}&=-\frac{\sigma_1}{2}\alpha+\frac{\sigma_1\sigma_0}{2}b+\frac{\sigma_1^2}{2}bx\\
&=\frac{\sigma_1 }{2}b\PAR{\sigma_0+\sigma_1 x}-\frac{\sigma_1 }{2}=\frac{\sigma_1 }{2z}\PAR{b-\alpha z}.
\end{align*}
We finally obtain
\begin{align*}
   & \frac{1}{2}\PAR{\sigma_0+\sigma_1 x}^2u''(x)+\PAR{c_0+c_1x}u'(x)-qu(x)=\frac{\sigma_1 ^2}{2}\alpha z^{1+a}\PAR{-ay(\alpha z)+\PAR{b-\alpha z}y'(\alpha z)+\alpha zy''(\alpha z)}
\end{align*}
which is equal to $0$ since $y$ is solution of {\bf E$(a,b)$}, and then $u$ is a solution of the ODE~\eqref{eq:ODE-sigma-affine}.

\medskip
\paragraph{\bf When $\sigma^2$ is affine}
We only study the case $c_1>0$, the proof for $c_1<0$ being similar. We will prove that 
\[
u(x)=\mathrm{e}^{-2\frac{c_1}{\sigma_1}x}z^by(\alpha z)
\]
with $a=\frac{q}{c_1}$, $b=1+2\frac{c_1\sigma_0-c_0\sigma_1}{\sigma_1^2}$, $\alpha=2\frac{c_1}{\sigma_1^2}$, $z=\sigma_0+\sigma_1 x$, and $y$ solution of the differential equation {\bf E$(1+a,1+b)$} (see \eqref{eq:ODE-hypergoem}), is a solution of
\begin{equation}\label{eq:ODE-sigma2-affine}
    \frac{1}{2}\PAR{\sigma_0+\sigma_1 x}u''+\PAR{c_0+c_1x}u'-qu=0.
\end{equation}

As $a>0$, from the integral form~\eqref{eq:U-integral} of $U$, we deduce that $\varphi_{q,F}$ defined by Equation~\eqref{eq:phi-sigma2-affine+} is solution of the ODE~\eqref{eq:ODE-sigma2-affine} and satisfies $\varphi_{q,F}(0)=1$ and $\lim_{x\to+\infty}\varphi_{q,F}(x)=0$, and the expression~\eqref{eq:psi-sigma2-affine+} of second fundamental solution $\psi_{q}$ follows.

Let us now prove that $u(x)=\mathrm{e}^{-2\frac{c_1}{\sigma_1}x}z^by(\alpha z)$, with $z,a, b,\alpha$ defined above is a solution of the ODE~\eqref{eq:ODE-sigma2-affine}. We compute the first and second derivative of $u$ with respect to $x$:
\begin{align*}
u'(x)&=-2\frac{c_1}{\sigma_1}\mathrm{e}^{-2\frac{c_1}{\sigma_1}x}z^by(\alpha z)+b\sigma_1\mathrm{e}^{-2\frac{c_1}{\sigma_1}x}z^{b-1}y(\alpha z)+\alpha\sigma_1\mathrm{e}^{-2\frac{c_1}{\sigma_1}x}z^by'(\alpha z)\\
u''(x)&=\PAR{4\frac{c_1^2}{\sigma_1^2}-4\frac{c_1b}{z}+\frac{b(b-1)}{z^2}\sigma_1^2}u(x)
+\alpha\PAR{-4c_1 z^b+2b\sigma_1^2 z^{b-1}}\mathrm{e}^{-2\frac{c_1}{\sigma_1}x}y'(\alpha z)+\alpha^2\sigma_1^2\mathrm{e}^{-2\frac{c_1}{\sigma_1}x}z^by''(\alpha z).
\end{align*}
Consequently,
\begin{align*}
& \frac{1}{2}\PAR{\sigma_0+\sigma_1 x}u''+\PAR{c_0+c_1x}u'-qu\\
&=\PAR{2\frac{c_1^2}{\sigma_1^2}z-2c_1b+\frac{b(b-1)}{2z}\sigma_1^2+\PAR{-2\frac{c_1}{\sigma_1}+\frac{b\sigma_1}{z}}(c_0+c_1x)-q}u(x)\\
&
+\alpha\SBRA{-2c_1 z+b\sigma_1^2 +\sigma_1(c_0+c_1 x)}\mathrm{e}^{-2\frac{c_1}{\sigma_1}x}z^{b}y'(\alpha z)+\frac{\alpha^2\sigma_1^2}{2}\mathrm{e}^{-2\frac{c_1}{\sigma_1}x}z^{b+1}y''(\alpha z).
\end{align*}
As $x=\frac{z-\sigma_0}{\sigma_1}$, we observe that
\begin{align*}
&2\frac{c_1^2}{\sigma_1^2}z-2c_1b+\frac{b(b-1)}{2z}\sigma_1^2+\PAR{-2\frac{c_1}{\sigma_1}+\frac{b\sigma_1}{z}}(c_0+c_1x)=-c_1,\\
&-2c_1 z+b\sigma_1^2 +\sigma_1(c_0+c_1 x)=-c_1z+\sigma_1^2+(c_1\sigma_0-c_0\sigma_1)=\frac{\sigma_1^2}{2}\PAR{1+b-\alpha z}.
\end{align*}
Thus,
\begin{align*}
     \frac{1}{2}\PAR{\sigma_0+\sigma_1 x}u''+\PAR{c_0+c_1x}u'-qu
     =c_1\mathrm{e}^{-2\frac{c_1}{\sigma_1}x}z^{b}\PAR{-(1+a)y(\alpha z)+\PAR{1+b-\alpha z}y'(\alpha z)+\alpha z y''(\alpha z)}=0,
\end{align*}
since $y$ is solution of {\bf E$(1+a,1+b)$}, and $u$ is solution of the ODE~\eqref{eq:ODE-sigma2-affine}.

\bibliographystyle{alpha}
\bibliography{references.bib}
\end{document}